\documentclass[a4paper, 11pt, reqno]{amsart}
\usepackage{amsmath}
\usepackage{amsfonts}
\usepackage{amsthm}
\usepackage{latexsym, amssymb}
\usepackage{epsfig}
\usepackage[pdftex]{hyperref}
\usepackage{indentfirst}
\usepackage{verbatim}
\usepackage[titletoc]{appendix}
\usepackage{url}

\newtheorem{thm}{Theorem}

\newtheorem*{teo}{Theorem}
\newtheorem{lem}{Lemma}
\newtheorem{rem}{Remark}
\newtheorem{prop}{Proposition}
\newtheorem{cor}{Corollary}
\newtheorem{defin}{Definition}

\newcommand{\C}{\mathbb{C}}
\newcommand{\N}{\mathbb{N}}

\newcommand{\de}{\delta}

\newcommand{\rd}{\partial}

\newcommand{\calM}{\mathcal{M}}
\newcommand{\e}{\mathcal{E}}
\newcommand{\ar}{\mathcal{A}}

\newcommand{\re}{\mathbb{R}}

\newcommand{\si}{\Sigma}

\newcommand{\ip}[2]{\bigl\langle #1 , #2 \bigr\rangle}

\newcommand{\ib}[2]{\biggl\langle #1 , #2 \biggr\rangle}

\DeclareMathOperator{\Ker}{Ker}

\newcommand{\ind}{\operatorname{ind}}
\newcommand{\n}{\operatorname{nul}}
\newcommand{\ed}{\operatorname{End}}
\newcommand{\jac}{\operatorname{Jac}}
\newcommand{\im}{\operatorname{im}}

\newenvironment{dem1}[1][\noindent \textit{Proof of Theorem \ref{t1}. }]{#1}{\hfill$\square$ \vspace{0.3cm}}

\newenvironment{dem2}[1][\noindent \textit{Proof of Theorem \ref{t2}. }]{#1}{\hfill$\square$ \vspace{0.3cm}}

\newenvironment{dem4}[1][\noindent \textit{Proof of Theorem \ref{t4}. }]{#1}{\hfill$\square$ \vspace{0.3cm}}

\title{Bounds for the Morse index of free boundary minimal surfaces}
\author{Vanderson Lima}
\address{Instituto de Matem\'atica e Estat\'istica\\ Universidade Federal do Rio Grande do Sul\\ Porto Alegre - RS}
\email{vanderson.lima@ufrgs.br}

\begin{document}

\maketitle

\begin{abstract}

Inspired by work of Ejiri-Micallef on closed minimal surfaces, we compare the energy index and the area index of a free-boundary minimal surface of a Riemannian manifold with boundary, and show that the area index is controlled from above by the area and the topology of the surface. Combining these results with work of Fraser-Li, we conclude that the area index of a free-boundary minimal surface in a convex domain of Euclidean three-space, is bounded from above by a linear function of its genus and its number of boundary components. We also prove index bounds for submanifolds of higher dimension.

\end{abstract}

\providecommand{\abs}[1]{\lvert#1\rvert}

\linespread{1} 

\section{Introduction}

Consider a minimal immersion $u: \si \to \mathcal{M}$ of a compact surface $\si$ into a Riemannian manifold $\calM$. We say that $u$ is {\it free-boundary}, if $u(\rd\si) \subset \mathcal{L}$ with $u(\rd\si)$ orthogonal to $\mathcal{L}$, for some embedded submanifold $\mathcal{L} \subset \calM$ of dimension greater than one. The {\it area index} of $u$ is defined as its Morse index when we see $u$ as critical point of the area functional for variations that map $\rd\si$ into $\mathcal{L}$. Roughly speaking, this quantity measures the maximal number of distinct local deformations that decrease the area to second-order. Unfortunately, in general this quantity is difficult to compute, and the best one can expect is to estimate it from above or below. On this work we are interested in upper bounds estimates.

A minimal immersion can also be seen as a critical point of the energy functional (a {\it harmonic map}), and conversely a conformal harmonic map is a minimal immersion. It turns out that this interesting fact is very useful: one way to obtain a minimal surface in some ambient space is first prove the existence of a harmonic map on each conformal class of a abstract surface, and then varying the conformal structure to obtain a conformal harmonic map. This approach comes from the solution of the Plateau problem \cite{Do,T.T}, and was extend to the case of closed surfaces \cite{S.Y,S.U1,S.U2,Z2}, and to free-boundary case \cite{M.Y,St,F1,C.F.P}, see also \cite{J2}.

On the other hand, a conformal harmonic map has also an Morse index associated to it, which we call the {\it energy index}. This quantity is somehow easier to compute than the {\it area index}, which can be explained by the fact that the second variation formula of energy has a expression simpler than the second variation formula of area. On this context, Ejiri and Micallef compared the two indices for a closed minimal surface (where $\mathcal{L} = \emptyset$) and obtained the following result:

\begin{teo}[Ejiri-Micallef, \cite{E.M}]
Let $u: \si^2 \to \calM^n$ be a minimal immersion of a closed oriented surface $\si$ of genus $g$, in a Riemannian manifold $\calM^n$, where $n\geq 3$. Then the area index $\ind_{\ar}(u)$, and the energy index $\ind_{\e}(u)$, satisfy the following:
\begin{enumerate}
\item $\ind_{\e}(u) \leq \ind_{\ar}(u) \leq \ind_{\e}(u) + \upsilon(g)$;\\
\item $\ind_{\e}(u) \leq c(\calM)|u(\si)|$;
\end{enumerate}
where $c(\calM) > 0$ is a constant depending only on $\calM$, and $$\upsilon(g) = \left\{
\begin{array}{rl}
0,& \textrm{if} \ \ g = 0,\\\\
1,& \textrm{if} \ \ g = 1,\\\\
6g - 6,& \textrm{if} \ \ g \geq 2.
\end{array} \right.$$
\end{teo}

\begin{rem}
In \cite{E.M}, the authors also allow the map to have branch points, and on this case the quantity $\upsilon$ also depends on the number of branch points. 
\end{rem}

The main goal of this work is to extend this previous result for the case of free-boundary minimal surfaces in Riemannian manifolds with boundary. Our first result is the following.

\begin{thm}\label{t1}
Let $u: (\si^2,\rd\si^2) \to (\mathcal{M}^n,\mathcal{L}^k)$ be a free-boundary minimal immersion, where $\si$ is a compact oriented surface of genus $g$ and $m$ boundary components, and $(\mathcal{M}^n,\mathcal{L}^k)$ are as above with $n \geq 3$ and $2 \leq  k \leq n-1$. Then the area index $\ind_{\ar}(u)$, and the energy index $\ind_{\e}(u)$, satisfy the following:
\begin{equation}\label{indine}
\ind_{\e}(u) \leq \ind_{\ar}(u) \leq \ind_{\e}(u) + \upsilon(g,m),
\end{equation}
where $\upsilon(g,m) = \left\{
\begin{array}{rl}
0,& \textrm{if} \ \ g=0 \ \textrm{and}\ m=1,\\\\
1,& \textrm{if} \ \ g = 0 \ \textrm{and}\ m=2,\\\\
6g - 6 + 3m,& \textrm{in the other cases}.
\end{array} \right.$
\end{thm}

To prove this result we follow the approach of \cite{E.M}, the idea is the following: given a deformation of a minimal immersion that decreases the area, if the second variation of area and energy along the respective variational vector fields do not coincide, then one tries to reparametrize the family of maps in order to obtain conformal maps, so that the new variations will coincide. Next one uses the Riemann-Roch theorem to count in how many ways one can do that. In our setting some difficulties arise due to the presence of the boundary, in particular we need a Riemann-Roch theorem on the case of surfaces with boundary. It turns out that there is a index-theoretic version of this theorem for Riemann surfaces with boundary, which is suitable to apply in our context.

Adapt the second part of the theorem of Ejiri-Micallef is more subtle. We follow the methods of Cheng and Tysk in \cite{C.T}, which are inspired by \cite{L.Y}. The idea is to control the index from above by the trace heat Kernel of an appropriated Schr\"odinger operator, related to the {\it Jacobi operator} of the immersion, and then to bound this kernel by the area. To do this, in \cite{C.T} the authors  plug the trace kernel in the Sobolev inequality for minimal surfaces. To our knowledge, there is no version of the Sobolev inequality for free-boundary minimal surfaces of general spaces. However, on the case of free-boundary hypersurfaces of some domain of the Euclidean space, Edelen obtained such inequality, see \cite{E}. We will adapt the work in \cite{E} and prove the free-boundary version of the Sobolev inequality for submanifolds (of any codimension) is valid in a large class of ambient spaces, see section \ref{sec4}. This result is of independent interest.

Following this we have.

\begin{thm}\label{t2}
Let $u: (\si^2,\rd\si^2) \to (\mathcal{M}^n,\rd\calM^n)$ be a free-boundary minimal immersion of an oriented compact surface $\si$, into a compact Riemannian manifold $\mathcal{M}^n$ with boundary $\rd \mathcal{M}^n$, where $n \geq 3$. Then, there is a constant $c = c(\mathcal{M}) > 0$ such that
\begin{equation}
\ind_{\e}(u) \leq c|u(\si)|
\end{equation}
and so,
\begin{equation}
\ind_{\ar}(u) \leq c|u(\si)| + \upsilon(g,m).
\end{equation}
\end{thm}

Since the {\it free-boundary Sobolev inequality} holds in great generality, the methods of \cite{C.T} also allows us to obtain bounds for the index of free-boundary submanifolds of higher dimension. 

\begin{thm}\label{t3}
Consider an compact oriented Riemannian manifold $\si^n$ with boundary $\rd\si$, and a compact Riemannian manifold $\mathcal{M}^{n+k}$ with boundary $\rd \mathcal{M}$, where $n\geq 3$ and $k \geq 1$. Let $u: (\si,\rd\si) \to (\mathcal{M},\rd\calM)$ be a free-boundary minimal immersion. Then there are constants $c_{l} = c(n,\calM)$, $l = 1,2$ such that
\begin{align}
\ind_{\e}(u) &\leq c_1\int_{\si}\Bigl(\max\bigl\{1,|\mathcal{R}|\bigr\}\Bigr)^{\frac{n}{2}}\, dA,\\
\ind_{\ar}(u) &\leq c_2\int_{\si}\Bigl(\max\bigl\{1,|\mathcal{R}| + |\mathrm{II}_{\si}|^2\bigr\}\Bigr)^{\frac{n}{2}}\, dA,
\end{align}
where $\mathcal{R}$ is defined in terms of the curvature of $\calM$ (equation \eqref{curv.oper}), and $\mathrm{II}_{\si}$ is the second fundamental form of $\si$.
\end{thm}

Using the same techniques we also obtain upper bounds for the relative Betti numbers of a free-boundary $\si$ minimal submanifold (see the end of section \ref{sec5}).

\begin{rem}
Examining the proof of the Sobolev inequality in \cite{E} and the proof of Theorems \ref{t2} and \ref{t3}, we see that the constant $c(\calM)$ only depends on the second fundamental form of $\partial\calM$ and the dimension $n$. 
\end{rem}

\begin{rem}
The conclusion of theorems \ref{t2} and \ref{t3} are still valid if $\calM$ is non-compact but satisfies the free-boundary Sobolev inequality (Definition \ref{FBSI}) and the norm of the second fundamental of $\rd\calM$ is bounded (see Theorem \ref{taux}).
\end{rem}

Combining the Theorem \ref{t2} with the area estimates in the work of Fraser and Li \cite{F.L} we obtain the following application.

\begin{thm}\label{t4}
Let $\mathcal{M}^3$ be a compact $3$-dimensional Riemannian manifold with nonnegative Ricci curvature and convex boundary. Let $\si \subset \mathcal{M}$ be an oriented properly embedded free boundary minimal surface of genus $g$ and $m$ boundary components. Then, there is a constant $c(\mathcal{M}) > 0$ such that the area index satisfies
$$\ind_{\ar}(\si) \leq c(\mathcal{M})(g+m).$$
\end{thm}

This previous theorem is interesting because on such ambient spaces there are existence results: one has the existence of free-boundary minimal disks and annuli, \cite{G.J,M.N.S} and of surfaces of controlled topology \cite{DL.R}. Also, extending the Almgreen-Pitts min-max theory to the free-boundary setting and using the methods of \cite{M.N}, Li and Zhou proved in \cite{L.Z} that manifolds satisfying the conditions of the previous theorem contains infinitely many (geometrically distinct) free-boundary minimal surfaces. Finally, on the special case of the Euclidean $3$-ball, we have various examples with known topology, see \cite{F.S1,F.S2,F.P.Z,K,K.L,K.W}.

On the other hand, in the case of convex domains of $\mathbb{R}^3$ we have lower bounds for the {\it area index}:

\begin{teo}[Ambrozio-Carlotto-Sharp \cite{A.C.S3}; Sargent \cite{S}]
Let $\si$ be a orientable properly embedded free boundary minimal surface of genus $g$ and $m$ boundary components, on a convex domain of $\mathbb{R}^3$. Then the area index satisfies
$$\frac{1}{3}(2g + m -1) \leq \ind_{\ar}(\si).$$
\end{teo}

This last theorem together with Theorem \ref{t4} shows that in convex domains of $\mathbb{R}^3$ the area index of free-boundary minimal surfaces is controlled above and below only by the topology of the surface. For other results about estimates for the area index, see \cite{F1,C.F} on the case of disks, and \cite{S.Z,T,D,S.S.T.Z} for the case of annuli (in particular, the critical catenoid).

Another motivation for index estimates comes from some recent works, where is showed that considering minimal hypersurfaces with bounded index we obtain compactness and finiteness results, see \cite{Sh,A.C.S1,H.Z,C.K.M,B.S,C} for the case of closed minimal hypersurfaces in closed manifolds, and \cite{A.C.S4,G.Z} for the case of free-boundary minimal hypersurfaces in manifolds with boundary. Combining the theorem B with the main results in \cite{A.C.S4,G.Z}, we obtain a compactness result in \ref{sec6.1}.\\

\noindent
{\bf Outline of the paper}: First, in section \ref{sec2} we fix some notations and recall a few facts about free-boundary minimal immersions and Riemann surfaces. In section \ref{sec3} we compare the second variations formulas of the area and the energy and use it to prove Theorem \ref{t1}. The section \ref{sec4} is devoted to prove the free-boundary Sobolev inequality. In section \ref{sec5} we prove Theorems \ref{t2} and \ref{t3}. We discuss some applications of the main results, including the Theorem \ref{t4}, in section \ref{sec6}. The Riemann-Roch theorem for Riemann surfaces with boundary is described in the appendix \ref{app}.
\\\\
{\bf Ackowledgements}: I wish to express my gratitude to Lucas Ambrozio for discussions, and for suggestions on the manuscript. I also thank Ivaldo Nunes and Harold Rosenberg for their interest on this work. I am grateful to Vladimir Medvedev for pointing out a correction on a earlier version of the manuscript, concerning the application of the Riemann-Roch Theorem. Finally, I thank the anonymous referee by suggestions and corrections.

\section{Preliminary}\label{sec2}

A good reference for the machinery of complex manifolds used in this section, is the section $3$ in the chapter $I$ of \cite{W}.

Denote by $\si^{\ell}$ a compact Riemannian manifold with boundary $\rd\si$, where $\ell \geq 2$.
Consider a complete oriented Riemannian manifold $\mathcal{M}^n$, and $\mathcal{L}^k$ an embedded submanifold of $\mathcal{M}$, where $n > \ell$ and $2\leq k \leq n-1$. We denote the metric and the Levi-Civita connection of $\mathcal{M}$ by $\ip{\cdot}{\cdot}$ and $\nabla^{\calM}$ respectively. The second fundamental form of $\mathcal{L}$ is defined by 
$$\mathrm{II}_{\mathcal{L}}(X,Y) = (\nabla_{X}^{\calM}Y)^{\perp},$$
where $\perp$ denotes the projection on the normal bundle of $\mathcal{L}$.

Let $u: (\si,\rd\si) \to (\mathcal{M},\mathcal{L})$ be a free-boundary minimal immersion. We have the decomposition $E := u^{*}(T\mathcal{M}) = u_{*}(T\si)\oplus N\si$, where $N\si$ is the normal bundle of $u$. The following connections are induced by $\nabla^{\calM}$
\begin{eqnarray*}
\nabla &:& \Gamma(E) \to \Gamma(E\otimes T^{*}\si),\\
D &:& \Gamma(T\si) \to \Gamma(T\si\otimes T^{*}\si),\\
\nabla^{\perp} &:& \Gamma(N\si) \to \Gamma(N\si\otimes T^{*}\si).
\end{eqnarray*}

Consider the area (volume) functional
\begin{equation}
\ar(u) = \int_{\si}\jac(u)\,dA,
\end{equation}
and the energy functional
\begin{equation}
\e(u) = \int_{\si}e(u)\, dA,
\end{equation}
where $dA$ is the volume form of the metric on $\si$ induced by $u$, $\jac(u) = \sqrt{\det\bigl(u_{*}^{t}\circ u_{*}\bigr)}$, and $e(u) = \displaystyle\frac{1}{2}\sum_{l=1}^{n}|u_{*}e_l|^2$, with $\{e_1,\ldots,e_n\}$ a (locally defined) orthonormal frame field on $\si$. 

If $\Sigma$ has dimension $2$, we have $\jac(u) \leq e(u)$. Therefore
\begin{equation}\label{aeine}
\ar(u) \leq \e(u),
\end{equation}
and the equality holds if, and only if, the map $u$ is conformal, i.e., $$\ip{u_{*}X}{u_{*}Y} = \phi\,\ip{X}{Y},$$ 
for some positive function $\phi \in C^{\infty}(\si)$, and $\forall \ X, Y \in \Gamma(T\si)$.

\begin{defin}
We say that $V \in \Gamma(E)$, $\xi \in \Gamma\bigl(N\si\bigr)$ are {\it admissible variations} if $V,\xi \in T_{u(p)}\mathcal{L}$, for all $p \in \rd\si$. 
\end{defin}

{\it The second variation formulas} of the area and the energy of a conformal harmonic map along admissible variations are given respectively by 
\begin{align*}
(\de^2 \ar)(\xi) &= \int_{\si} \Bigl(
|\nabla^{\perp}\xi|^2 - |(\nabla \xi)^{\top}|^2 - 
\ip{\mathcal{R}(\xi)}{\xi}\Bigr)dA + \int_{\rd \si} \ip{\mathrm{II}_{\mathcal{L}}(\xi,\xi)}{\eta}\, dL, \\
(\de^2 \e)(V) &= \int_{\si} \Bigl(|\nabla V|^2 - \ip{\mathcal{R}(V)}{V}\Bigr)dA + \int_{\rd \si} \ip{\mathrm{II}_{\mathcal{L}}(V,V)}{\eta}\, dL, 
\end{align*}
where $\eta$ is the outward pointing unit normal of $\rd\si$, and $\mathcal{R}(V)$ is given by
\begin{equation}\label{curv.oper}
\mathcal{R}(V) := \sum_{l=1}^{n} R\bigl(V,u_{*}e_l\bigr)u_{*}e_l,
\end{equation}
where $\{e_1,\cdots,e_n\}$ is a (locally defined) orthonormal frame field on $\si$ and $R$ is the curvature tensor of $\mathcal{M}$.

\noindent
\begin{defin}
The Morse index of $\e(u)$, denoted by $\ind_{\e}(u)$, is the maximal dimension of a subspace of $\Gamma(E)$ consisting of admissible variations, on which the second variation of $\e$ is negative. The nullity of $\e(u)$, denoted by $\n_{\e}(u)$, is the maximal dimension of a subspace of $\Gamma(E)$ consisting of admissible variations, on which the second variation of $\mathcal{E}$ is zero. Analogously we define the index $\ind_{\ar}(u)$ and the nullity $\n_{\ar}(u)$ of $\ar(u)$.
\end{defin}

Suppose $\si$ is a Riemann surface with a complex structure $J: T\si \to T\si$. Consider $T_{\C}\si = T\si\otimes_{\re}\C$ and $E_{\C} = u^{*}(T_{\C}\mathcal{M})$, where $T_{\C}\mathcal{M} = T\mathcal{M}\otimes_{\re}\C$. Extend $\ip{\cdot}{\cdot}$ complex bilinearly to $E_{\C}$, and $J$ complex linearly to $T_{\C}\si$. We have the decompositions
\begin{equation}\label{spl}
T_{\C}\si = T^{1,0}\si\oplus T^{0,1}\si \ \ \textrm{and} \ \ 
T_{\C}^{*}\si = \Lambda^{1,0}\si\oplus\Lambda^{0,1}\si,
\end{equation}

\noindent
which are orthogonal with respect to $\ip{\cdot}{\cdot}$, where
$$T^{1,0}_{p}\si = \{v \in (T_{\C}\si)_p/ J(v) = iv\},$$
and
$$T^{0,1}_{p}\si = \{v \in (T_{\C}\si)_p/ J(v) = -iv\},$$
and $\Lambda^{1,0}\si$ and $\Lambda^{0,1}\si$ are the $\C$-dual bundles of $T^{1,0}\si$ and $T^{0,1}\si$ respectively.

These constructions allow us to define the operators
$$\rd: \Gamma(\mathcal{T}) \to \Gamma(\mathcal{T}\otimes\Lambda^{1,0}\si) \ \ \textrm{and} \ \ \bar{\rd}: \Gamma(\mathcal{T}) \to \Gamma(\mathcal{T}\otimes\Lambda^{0,1}\si),$$
on every holomorphic bundle $\Pi: \mathcal{T} \to \si$.

On extending $\nabla$ complex linearly to $\Gamma(E_{\C})$, and using the splitting \eqref{spl}, we obtain
\begin{eqnarray*}
\nabla^{1,0} : \Gamma(E_{\C}) \to \Gamma(E_{\C}\otimes \Lambda^{1,0}\si),\quad D^{1,0} : \Gamma(T_{\C}\si) \to \Gamma(T_{\C}\si\otimes \Lambda^{1,0}\si).
\end{eqnarray*}
In a similar way, we obtain $\nabla^{0,1}$ and $D^{0,1}$.

\section{The comparison between the indices}\label{sec3}

Let $\si$ be a compact surface with boundary. Using equation \eqref{aeine} is easy to prove that a variation which decreases the energy of a conformal harmonic map $u: (\si,\rd\si) \to (\mathcal{M}^n,\mathcal{L}^k)$ must also decrease its area, so 
\begin{equation}
\ind_{\e}(u) \leq \ind_{\ar}(u).
\end{equation} 
In \cite{E.M}, the authors find a condition which guarantees that a variation which decreases the area of a conformal harmonic map will also decrease the energy, in the case of immersions of closed surfaces. The idea is reparametrize the variation so as to maintain it conformal with respect to the initial conformal structure. Supposing there is a map from $\Gamma(N\si)$ to $\Gamma\bigl(u_{*}(T\si)\bigr)$ such that
\begin{enumerate}
\item the map $\xi \to X_{\xi}$ is linear;
\item the family of maps corresponding to the variation vector field $\xi + X_{\xi}$ is a family of conformal maps;
\end{enumerate}
which gives $(\de^2 \e)(\xi + X_{\xi}) = (\de^2 \ar)(\xi)$, they proved that the following holds
\begin{equation}\label{maineq}
D^{1,0}X^{0,1}_{\xi} = - (\nabla^{1,0}\xi)^{\top}.
\end{equation}

\noindent
They also obtained a converse of this, comparing the numbers $(\de^2 \e)(\xi + X)$ and $(\de^2 \ar)(\xi)$.

In our case, we have to impose the additional condition that $\xi$ and $\xi + X_{\xi}$ are admissible variations. Also, the boundary terms coming for the second variation formulas will affect the calculations. The idea is then to consider a variation that with an appropriate boundary condition along $\rd\si$. Since the calculations in \cite{E.M} necessary to derive equation \eqref{maineq} are local, it also holds on the case with boundary, under the additional conditions. The connverse is proved in the next result, which is an analogous of theorem $2.1$ in \cite{E.M} in the context of free boundary minimal surfaces. For related results on the case the ambient space is the Euclidean unit ball, see section $6$ of \cite{F.S2}.

\begin{thm}\label{indcomp}
Let $u: (\si^2,\rd\si^2) \to (\mathcal{M}^n,\mathcal{L}^k)$ be a free-boundary minimal immersion, where $\si$ is a compact Riemann surface with boundary, and $(\mathcal{M}^n,\mathcal{L}^k)$ are as in section \ref{sec2}. Let $X \in \Gamma\bigl(u_{*}(T\si)\bigr)$ and $\xi \in \Gamma(N\si)$ be admissible variations. Then
$$(\de^2 \ar)(\xi) \leq (\de^2 \e)(X + \xi),$$
and the equality holds if, and only if,
$$D^{1,0}X^{0,1} = - (\nabla^{1,0}\xi)^{\top}.$$
\end{thm}

\begin{proof}
By abuse of notation we denote the image $u(\si)$ by $\si$. Let $x,y$ be local isothermal coordinates which cover $\si$ and $\rd \si$ up to sets of measure zero, and let $z = x + iy$ be the corresponding local complex coordinate. We can suppose that for any $p \in \rd\si$ we have $\rd_x \in T_{p}(\rd\si)$ and $\rd_y \perp T_{p}(\rd\si)$. We define
\begin{align*}
\rd_{z} &= \frac{1}{2}\bigl(\rd_x - i\rd_y\bigr), \quad \rd_{\bar z} = \frac{1}{2}\bigl(\rd_x + i\rd_y\bigr),\quad
dz = dx + idy, \quad d\bar{z} = dx - idy,
\end{align*}
and $u_z = u_{*}(\rd_{z})$, $u_{\bar z} = u_{*}(\rd_{\bar z})$. Locally, the fibre of $u_{*}\big(T^{1,0}\si\big)$ is spanned by $u_{z}$, and the fibre of $u_{*}\big(T^{0,1}\si\big)$ is spanned by $u_{\bar z}$. 

There is $\lambda > 0$ such that the metric on $\si$ is given by
$$\lambda^2\big(dx^2 + dy^2\big) = \lambda^2dzd\bar{z},$$
and on these coordinates $D^{1,0} = (D_{\rd_z})\otimes dz$, $D^{0,1} = (D_{\rd_{\bar{z}}})\otimes d\bar{z}$, $\nabla^{1,0} = (\nabla_{\rd_z})\otimes dz$ and $\nabla^{0,1} = (\nabla_{\rd_{\bar{z}}})\otimes d\bar{z}$. Moreover, we can parametrize $\rd\si$ as $x(s) = s,\, y(s) = 0$, $s\in I$.\\

Let $V = X + \xi$. Then $(\de^2 \e)(V)$ is given by

\begin{equation}\label{3.2}
\begin{split}
  & \int_{\Sigma} \Bigl(|\nabla_{\rd_x}V|^2 +
  |\nabla_{\rd_y}V|^2 - \ip{R(V,u_x)u_x}{V} - \ip{R(V,u_y)u_y}{V} 
  \Bigr)dxdy\\
  & + \int_{I} \ip{\mathrm{II}_{\mathcal{L}}(V,V)}{\eta} \lambda\,ds.
  \end{split}
\end{equation}

We first deal with the integrand on the first integral in \eqref{3.2}. Since $V$ is a real section 
\begin{equation*}\label{3.3} 
|\nabla_{\rd_x}V|^2 + 
  |\nabla_{\rd_y}V|^2 = 4\,|\nabla_{\rd_z}V|^2. 
\end{equation*}

\noindent 
We have
\begin{equation}\label{3.4} 
 \nabla_{\rd_z} V = (\nabla_{\rd_z} V)^{\perp} + \zeta + 
  (\nabla_{\rd_z} X^{1,0})^{\top} 
\end{equation} 
where 
\begin{equation}\label{3.5} 
  \zeta := (\nabla_{\rd_z}\xi)^{\top} + 
  D_{\rd_z}X^{0,1} \, . 
\end{equation}

Since $u$ is conformal we have
\begin{equation}\label{conf}
\ip{\nabla_{\rd_z}X^{1,0}}{u_z} = 0 \ \ \textrm{and} \ \
\ip{\nabla_{\rd_z}X^{0,1}}{u_{\bar z}} = 0.
\end{equation}
Also, since $u$ is harmonic we obtain
\begin{equation}\label{harm}
\ip{\nabla_{\rd_z}\xi}{u_{\bar z}} = - \ip{\xi}{\nabla_{\rd_z}u_{\bar z}} = 0.
\end{equation}

Using \eqref{conf} and \eqref{harm} in \eqref{3.4} we obtain 
\begin{equation}\label{3.6}
  |\nabla_{\rd_z} V|^2 = |(\nabla_{\rd_z} V)^{\perp}|^2 + |\zeta|^2 + 
  |(\nabla_{\rd_z} X^{1,0})^{\top}|^2. 
\end{equation}

\noindent
Locally, we can write
$X^{0,1} = \psi\,u_{\bar{z}}$ for some function $\psi$. So it follows that
\begin{equation}\label{3.7} 
  (\nabla_{\rd_z} X^{0,1})^{\perp} = 0. 
\end{equation} 
So, \eqref{3.7} allows us to re-write \eqref{3.6} as 
\begin{equation}\label{3.8} 
\begin{split} 
  |\nabla_{\rd_z} V|^2 &= |(\nabla_{\rd_z} \xi)^{\perp}|^2 + |\zeta|^2 + 
  |\nabla_{\rd_z} X^{1,0}|^2 \\ 
  & \quad + 
  \ip{(\nabla_{\rd_z} \xi)^{\perp}}{\nabla_{\rd_{\bar{z}}} X^{0,1}} 
  +
  \ip{(\nabla_{\rd_{\bar{z}}} \xi)^{\perp}}{\nabla_{\rd_z} X^{1,0}}. 
\end{split} 
\end{equation}

For the remaining terms of \eqref{3.8} we have 
\begin{equation*}\label{3.9} 
\begin{split} 
\ip{(\nabla_{\rd_z}\xi)^{\perp}}{\nabla_{\rd_{\bar{z}}}X^{0,1}} 
& = \rd_z \ip{\xi}{\nabla_{\rd_{\bar{z}}}X^{0,1}} - 
\ip{\xi}{\nabla_{\rd_z}\nabla_{\rd_{\bar{z}}} X^{0,1}} \\
\ip{(\nabla_{\rd_{\bar{z}}}\xi)^{\perp}}{\nabla_{\rd_{z}}X^{1,0}}  & = \rd_{\bar{z}} \ip{\xi}{\nabla_{\rd_z}X^{1,0}} - 
\ip{\xi}{\nabla_{\rd_{\bar{z}}}\nabla_{\rd_z} X^{1,0}}. 
\end{split} 
\end{equation*} 

\noindent
However, from \eqref{3.7} and \eqref{3.5} we have 
$\nabla_{\rd_z}X^{0,1} = (\nabla_{\rd_z}X^{0,1})^{\top} = 
\zeta - (\nabla_{\rd_z}\xi)^{\top}$, hence 
\begin{equation}\label{3.10} 
\begin{split} 
  \ip{\nabla_{\rd_z}\nabla_{\rd_{\bar{z}}}X^{0,1}}{\xi} 
  &= \ip{R(u_z,u_{\bar{z}})X^{0,1}}{\xi} + 
  \ip{\nabla_{\rd_{\bar{z}}}(\zeta - (\nabla_{\rd_z}\xi)^{\top})}{\xi} \\ 
  &= \ip{R(u_z,u_{\bar{z}})X^{0,1}}{\xi} + 
  |(\nabla_{\rd_z}\xi)^{\top}|^2 - 
  \ip{\zeta}{(\nabla_{\rd_{\bar{z}}}\xi)^{\top}}. 
\end{split} 
\end{equation}

\noindent 
In a similar way we can prove that 
\begin{equation}\label{3.11} 
\ip{\nabla_{\rd_{\bar{z}}}\nabla_{\rd_z}X^{1,0}}{\xi} = 
\ip{R(u_{\bar{z}},u_z)X^{1,0}}{\xi} + 
  |(\nabla_{\rd_{\bar{z}}}\xi)^\top|^2 - 
  \ip{\bar{\zeta}}{(\nabla_{\rd_z}\xi)^{\top}}. 
\end{equation}

Let $T$ be the unit tangent to $\rd\si$ in the chosen orientation. Then
$$T = \lambda^{-1}\,u_x,\quad \eta = - \lambda^{-1}\,u_y.$$
Locally $X^{1,0} = \phi\,u_{z}$ and $X^{0,1} = \overline{\phi}\,u_{z}$. Along $\rd\si$ we have $X|_{\rd\si} = f\,u_x$, where $f$ is real, so $f = \phi|_{\rd\si} = \overline{\phi}|_{\rd\si}$. Moreover $dx\wedge dy = \frac{i}{2}dz\wedge d\bar{z}$. Using this and Stokes' Theorem we obtain
\begin{equation}\label{neweqbdry}
\begin{split}
&\quad 4\int_{\si}\Big[\rd_z \ip{\xi}{\nabla_{\rd_{\bar{z}}}X^{0,1}} + \rd_{\bar{z}} \ip{\xi}{\nabla_{\rd_z}X^{1,0}}\Big]dxdy\\ 
&= -\, 2i\int_{\si}\Big[\rd_z \ip{\nabla_{\rd_{\bar{z}}}\xi}{X^{0,1}} + \rd_{\bar{z}} \ip{\nabla_{\rd_z}\xi}{X^{1,0}}\Big]dz\wedge d\bar{z}\\
&= -\, 2i\int_{\rd\si} \Big[\ip{\nabla_{\rd_{\bar{z}}}\xi}{X^{0,1}}d\bar{z} - \ip{\nabla_{\rd_{z}}\xi}{X^{1,0}}dz\Big]\\
&= 2\int_{\rd\si}\phi\ip{\nabla_{\rd_{x}}\xi}{u_{y}}\,dx\\
&= -2\int_{I}\ip{\nabla_{X}\xi}{\eta}\lambda\,ds.
\end{split}
\end{equation}

Substituting \eqref{3.9}, \eqref{3.10}, \eqref{3.11} and \eqref{neweqbdry} in \eqref{3.8} we obtain
\begin{equation}\label{3.12} 
\begin{split} 
  \int_{\si} |\nabla_{\rd_z}V|^2 \,dxdy & = \int_{\si} 
  \left(|(\nabla_{\rd_z}\xi)^{\perp}|^2 + |\zeta|^2 + 
  |\nabla_{\rd_z}X^{1,0}|^2 - 
  2\,|(\nabla_{\rd_z}\xi)^{\top}|^2 \right. \\ 
  & \quad - \ \ip{R(u_z,u_{\bar{z}})X^{0,1}}{\xi} - 
  \ip{R(u_{\bar{z}},u_z)X^{1,0}}{\xi} \\ 
  & \left. \quad + \ \ip{\zeta}{(\nabla_{\rd_{\bar{z}}}\xi)^{\top}} 
  + \ip{\bar{\zeta}}{(\nabla_{\rd_z}\xi)^{\top}}\right)dxdy \\
  & \quad -\, 2\int_{I}\ip{\nabla_{X}\xi}{\eta}\lambda\,ds. 
\end{split} 
\end{equation}

It only remains to see what happens with the last two terms in \eqref{3.2}: 
\begin{equation}\label{3.13} 
  \begin{split} 
  \ip{R(V,u_x)u_x}{V} + \ip{R(V,u_y)u_y}{V} 
  = 4\,\ip{R(V,u_z)u_{\bar{z}}}{V} \\ 
  = 4\Bigl(\ip{R(\xi,u_z)u_{\bar{z}}}{\xi} + 
  \ip{R(X^{0,1},u_z)u_{\bar{z}}}{X^{1,0}} \\ 
  \qquad + \ip{R(\xi,u_z)u_{\bar{z}}}{X^{1,0}} + 
  \ip{R(X^{0,1},u_z)u_{\bar{z}}}{\xi} \Bigr). 
  \end{split} 
\end{equation}

\noindent 
By the second Bianchi identity, for any $W \in \Gamma(E)$ we have
\begin{equation}\label{3.14} 
  \begin{split} 
\ip{R(X^{0,1},u_z)u_{\bar{z}}}{W} + 
\ip{R(u_z,u_{\bar{z}})X^{0,1}}{W} &= 0, \\ 
\ip{R(X^{1,0},u_{\bar{z}})u_z}{W} + 
\ip{R(u_{\bar{z}},u_z)X^{1,0}}{W} &= 0. 
  \end{split} 
\end{equation}

On the other hand, on $\rd \si$ we have
\begin{equation}\label{bdrterm}
 \ip{\mathrm{II}_{\mathcal{L}}(V,V)}{\eta} =  \ip{\nabla_X X + 2\nabla_X \xi + \mathrm{II}_{\mathcal{L}}(\xi,\xi)}{\eta}.
\end{equation}

Using \eqref{3.3}, \eqref{3.12}, \eqref{3.13}, \eqref{3.14} and \eqref{bdrterm} in 
\eqref{3.2} yields: 
\begin{equation}\label{3.15} 
  \begin{split} 
  (\de^2 \e)(V) &= 4 \int_{\si} \left( 
  |(\nabla_{\rd_z}\xi)^{\perp}|^2 + |\nabla_{\rd_z}X^{1,0}|^2 + 
  |\zeta|^2 - 2\,|(\nabla_{\rd_z}\xi)^{\top}|^2 \right. \\ 
  & \quad - \ip{R(\xi,u_z)u_{\bar{z}}}{\xi} - 
  \ip{R(X^{0,1},u_z)u_{\bar{z}}}{X^{1,0}} \\ 
  & \left. \quad + \ \ip{\zeta}{(\nabla_{\rd_{\bar{z}}}\xi)^{\top}} 
  + \ip{\bar{\zeta}}{(\nabla_{\rd_z}\xi)^{\top}} \right)dxdy \\
  & \quad + \int_{I} \ip{\nabla_X X + \mathrm{II}_{\mathcal{L}}(\xi,\xi)}{\eta}\lambda\,ds. 
  \end{split} 
\end{equation}

On the other hand,
\begin{equation}\label{3.16}
  \begin{split} 
  \int_{\si} |\nabla_{\rd_z}X^{1,0}|^2\,dxdy
  &= \int_{\si} \ip{\nabla_{\rd_z}X^{1,0}}{\nabla_{\rd_{\bar z}}X^{0,1}}\, dxdy \\
  &= \int_{\si} \Bigl(\rd_{\bar z}\ip{\nabla_{\rd_z}X^{1,0}}{X^{0,1}} - \ip{\nabla_{\rd_{\bar z}}\nabla_{\rd_z}X^{1,0}}{X^{0,1}}\Bigr) dxdy\\
  &= \int_{\si} \Bigl(\rd_{\bar z}\ip{\nabla_{\rd_z}X^{1,0}}{X^{0,1}} - \ip{\nabla_{\rd_z}\nabla_{\rd_{\bar z}}X^{1,0}}{X^{0,1}}\\
  & \quad + \ \ip{R(u_z,u_{\bar{z}})X^{1,0}}{X^{0,1}}\Bigr) dxdy\\
  &= \int_{\si} \Bigl(\rd_{\bar z}\ip{\nabla_{\rd_z}X^{1,0}}{X^{0,1}} - \rd_z\ip{\nabla_{\rd_{\bar z}}X^{1,0}}{X^{0,1}}\\
  & \quad + \ |\nabla_{\rd_{\bar z}}X^{1,0}|^2 + \ip{R(u_z,u_{\bar{z}})X^{1,0}}{X^{0,1}}\Bigr) dxdy.
  \end{split}  
 \end{equation}
 
Also, by Stokes' Theorem,
\begin{equation}\label{neweqbdr2}
\begin{split}
 &\quad\, 4\int_{\si} \Bigl(\rd_{\bar z}\ip{\nabla_{\rd_z}X^{1,0}}{X^{0,1}} - \rd_z\ip{\nabla_{\rd_{\bar z}}X^{1,0}}{X^{0,1}}\Big)dxdy\\
 &= 2i\int_{\si} \Bigl(\rd_{\bar z}\ip{\nabla_{\rd_z}X^{1,0}}{X^{0,1}} - \rd_z\ip{\nabla_{\rd_{\bar z}}X^{1,0}}{X^{0,1}}\Big)dz\wedge d\bar{z}\\
 &= -2i\int_{\rd\si} \Bigl(\ip{\nabla_{\rd_z}X^{1,0}}{X^{0,1}}dz + \ip{\nabla_{\rd_{\bar z}}X^{1,0}}{X^{0,1}}d\bar{z}\Big)\\
 &= -2i\int_{\rd\si} \ip{\nabla_{\rd_x}X^{1,0}}{X^{0,1}}\,dx\\
 &= -\frac{i}{2}\int_{\rd\si} \Big[\ip{\nabla_{\rd_x}(\phi\,u_x)}{\phi\,u_x} + \ip{\nabla_{\rd_x}(\phi\,u_y)}{\phi\,u_y}\Big]\,dx\\
 &\quad + \frac{1}{2}\int_{\rd\si} \Big[\ip{\nabla_{\rd_x}(\phi\,u_x)}{\phi\,u_y} - \ip{\nabla_{\rd_x}(\phi\,u_y)}{\phi\,u_x}\Big]\,dx\\
 &= -\frac{i}{2}\int_{\rd\si} \rd_x(\lambda^2)\,dx + \int_{\rd\si} \ip{\nabla_{(\phi\,\rd_x)}(\phi\,u_x)}{u_y}\,dx\\
 &= - \int_{I} \ip{\nabla_{X}X}{\eta}\lambda\,ds.
  \end{split}  
 \end{equation}

Using the equation \eqref{3.5} and the orthogonal decomposition \eqref{spl} a simple calculation shows that 
\begin{equation}\label{3.17} 
  |\nabla_{\rd_{\bar z}}X^{1,0}|^2
  = |\zeta|^2 + |(\nabla_{\rd_z}\xi)^{\top}|^2 - 
  \ip{\zeta}{(\nabla_{\rd_{\bar{z}}}\xi)^{\top}}  -  \ip{\bar{\zeta}}{(\nabla_{\rd_z}\xi)^{\top}}.
\end{equation} 

\noindent
Substituting \eqref{3.16}, \eqref{neweqbdr2} and \eqref{3.17} in \eqref{3.15}, and using \eqref{3.14} we obtain 
\begin{equation*}\label{comp} 
\begin{split} 
(\de^2 \e)(V) &= 4 \int_{\si} \left(
|\nabla^{\perp}_{\rd_z}\xi|^2 - |(\nabla_{\rd_z}\xi)^{\top}|^2 - 
\ip{R(\xi,u_z)u_{\bar{z}}}{\xi} + 2|\zeta|^2 \right)dxdy \\
& \quad + \int_{a}^{b} \ip{\mathrm{II}_{\mathcal{L}}(\xi,\xi)}{\eta}\lambda\,ds\\
&= \int_{\si} \left(
|\nabla^{\perp}\xi|^2 - |(\nabla \xi)^{\top}|^2 - 
\ip{\mathcal{R}(\xi)}{\xi}\right)dA \\
& \quad + \int_{\rd \si} \ip{\mathrm{II}_{\mathcal{L}}(\xi,\xi)}{\eta}\, dL + 8\int_{\si} |\zeta|^2\, dxdy \\
&= (\de^2 \ar)(\xi) + 8 \int_{\si} \bigl|D^{1,0}X^{0,1} + (\nabla^{1,0}\xi)^{\top}\bigr|^2 \,dA. 
\end{split} 
\end{equation*} 
Thus, the result follows.
\end{proof}

Now we can prove the Theorem \ref{t1}.\\

\begin{dem1}
The first inequality was already discussed. Let us thus discuss the second one.

Let $F$ be the subbundle of $T^{0,1}\si|_{\rd\si}$ whose sections in local isothermal coordinates are given by $\psi\,\partial_{\bar{z}}$, where $\psi$ is a real function. Consider a maximal subspace $S$ on which $\de^2 \ar < 0$, and $\xi \in S$. 

By the Fredholm alternative, the boundary-value problem 
\begin{equation}\label{obst}
\left\{
\begin{array}{rl}
&D^{1,0}X^{0,1} =\ - (\nabla^{1,0}\xi)^{\top}, \quad \textrm{on} \ \si\\\\
&X^{0,1}(p) \in F, \quad \textrm{if}\ p\in \rd\si
\end{array} \right.
\end{equation} 
has a solution if, and only if, $(\nabla^{1,0}\xi)^{\top}$ is orthogonal to $\ker (D^{1,0}){^*}$, where $(D^{1,0}){^*}$ is the adjoint of 
$$D^{1,0}: \Gamma_{F}(T^{0,1}\si) \to \Gamma(T^{0,1}\si\otimes\Lambda^{1,0}\si).$$ 

An calculation (we can proceed as remark C.1.3 on page $580$ of \cite{Mc.Sa}) shows that 
$$(D^{1,0}){^*} = -*\bar{\rd}* = i*\bar{\rd} \colon \Gamma_{F}(T^{0,1}\si \otimes \Lambda^{1,0}\si) \to \Gamma(T^{0,1}\si),$$ 
where $*$ is the Hodge star operator, and the second equality uses the fact that $*\omega = J\omega = -i\omega$, for $\omega \in \Gamma(T^{0,1}\si \otimes \Lambda^{1,0}\si)$.
Therefore 
$$\ker (D^{1,0}){^*} = H_{F}^{0}(T^{0,1}\si \otimes \Lambda^{1,0}\si),$$ 
where on the right hand side we have the space of holomorphic sections $Y$ of $T^{0,1}\si \otimes \Lambda^{1,0}\si$ such that $Y(p) \in F,\, \forall\, p \in \rd\si$. 

Denote
\begin{align*}
h^0(\Lambda^{0,1}\si) &= \ \textrm{complex dimension of} \ H_{F}^0(\Lambda^{0,1}\si),\\
h^0(T^{0,1}\si \otimes \Lambda^{1,0}\si) &= \ \textrm{complex dimension of} \ H_{F}^0(T^{0,1}\si \otimes \Lambda^{1,0}\si).
\end{align*}

\noindent
Then, we may find a subspace
$\tilde{S} \subset S$ of real dimension 
$$n \geq \dim S - 2h^0(T^{0,1}\si \otimes \Lambda^{1,0}\si)$$
for which \eqref{obst} has a solution $X_{\xi}^{0,1}$ whenever $\xi \in \tilde{S}$. Define $X_{\xi}^{1,0} = \overline{X_{\xi}^{0,1}}$ and $X_{\xi} = X_{\xi}^{1,0} + X_{\xi}^{0,1}$. Since the equation \eqref{obst} is linear, we may choose the map $\xi \to X_{\xi}$ to be linear. 
Let $\hat{S} = \{\xi + X_{\xi}/\, \xi \in \tilde{S} \} 
\subset \Gamma(E)$. Then, by Theorem \ref{indcomp}, we have
$\de^2 \e \vert_{\hat{S}} < 0$ and thus, 
$$\ind_{\e}(u) \geq \dim \hat{S} = \dim \tilde{S} \geq \ind_{\ar}(u) - 2h^0(T^{0,1}\si \otimes \Lambda^{1,0}\si).$$

Now, by the Riemann-Roch theorem for Riemann surfaces with boundary applied to the operator (see the example on the appendix \ref{ex})
$$\bar{\rd}: W_{F}^{m,q}(\Lambda^{0,1}\si) \to W^{m-1,q}(\Lambda^{0,1}\si\otimes\Lambda^{0,1}\si)$$
we have
$$2h^{0}(\Lambda^{0,1}\si) - 2h^0(T^{0,1}\si\otimes \Lambda^{1,0}\si) = 3\chi(\si).$$

By \eqref{disk}, if $\si$ is a disk, then $\bar{\rd}$ is surjective and $h^0(T^{0,1}\si \otimes \Lambda^{1,0}\si) = 0$. On the other hand, by \eqref{hyperb}, if $\chi(\si) < 0$, then $h^{0}(\Lambda^{0,1}\si) = 0$, so
$$2h^0(T^{0,1}\si \otimes \Lambda^{1,0}\si) = -3\chi(\si) = 6g - 6 + 3m.$$

Finally, if $\si$ is an annulus, then
\begin{equation}\label{genzero}
2h^{0}(\Lambda^{0,1}\si) = 2h^0(T^{0,1}\si\otimes \Lambda^{1,0}\si).
\end{equation}  
If $h^0(T^{0,1}\si\otimes \Lambda^{1,0}\si) = 0$, there is nothing to do. If $h^0(T^{0,1}\si\otimes \Lambda^{1,0}\si) > 0$, consider the double $\tilde{\si}$ of $\si$ endowed with a symmetric complex structure (as explained on the appendix \ref{ex}), so that $\tilde{\si}$ has genus $1$. Let $[d]$ be the divisor associated to the bundle $\Lambda^{0,1}\tilde{\si}$. Following the notation of \cite{F.K}, the dimension of this divisor is $r\bigl([d]\bigr) = 2h^{0}(\Lambda^{0,1}\si) > 0$, hence the divisor $[d]$ is special. Thus, by Clifford's theorem (see theorem III$.8.4$ of \cite{F.K}) and equation \eqref{genzero}
\begin{equation*}
2h^0(T^{0,1}\si\otimes \Lambda^{1,0}\si) = r\bigl([d]\bigr) \leq \frac{\deg\bigl([d]\bigr)}{2} + 1 = 1.
\end{equation*}
\end{dem1}

Define
\begin{align*}
\n_{\e}^T(u) &= \textrm{dimension of the space of purely tangential Jacobi fields of $u$,}\\ 
& \ \ \ \ \textrm{as a critical point of} \ \e.
\end{align*}
\noindent
A minor modification of the proof of Theorem \ref{t1} yields.\\

\begin{thm}\label{nulcomp} 
Let $u: (\si,\rd\si) \to (\mathcal{M},\mathcal{L})$ and $\upsilon(g,m)$ as in Theorem \ref{t1}. Then
\begin{equation}
\begin{split}
\ind_{\e}(u) + \n_{\e}(u) - \n_{\e}^T(u) &\leq \ind_{\ar}(u) + \n_{\ar}(u) \\
\leq \ind_{\e}(u) &+ \n_{\e}(u) - \n_{\e}^T(u) + \upsilon(g,m),
\end{split}
\end{equation}
which combined with \eqref{indine} gives us,
\begin{equation}
\n_{\e}(u) - \n_{\e}^T(u) - \upsilon(g,m) \leq \n_{\ar}(u) \leq \n_{\e}(u) - \n_{\e}^T(u) + \upsilon(g,m). 
\end{equation}
 \end{thm} 

\section{A Sobolev inequality for free boundary submanifolds}\label{sec4}

\begin{defin}\label{bound}
We say that a Riemannian manifold $\mathcal{M}$ satisfies the property $(\star)$ if there is a isometric immersion $f: \mathcal{M} \to \mathcal{N}$ such that:
\begin{enumerate}
\item The second fundamental form of the immersion satisfies
$$\sup_{\mathcal{N}}| \mathrm{II}_{f}| \leq C,$$ 
for some constant $C > 0$;
\item $\mathcal{N}$ is a Cartan-Hadamard manifold such that its sectional curvature satisfy 
$$sec_{\mathcal{N}} \leq \kappa \leq 0.$$
\end{enumerate}

\end{defin}

\begin{rem} 
Observe that every compact Riemannian manifold $\mathcal{M}$ satisfies the property $(\star)$, because by Nash theorem $\mathcal{M}$ can be embedded in some Euclidean space $\re^{n}$, so condition $2$ is valid, and condition $1$ is satisfied by compactness. 
\end{rem}


\begin{thm}\label{edel}
Consider an oriented complete Riemannian manifold $\mathcal{M}^{n+k}$ with boundary $\rd\calM$, which satisfies the property $(\star)$, where $n\geq 2,\,k \geq 1$. Let $u:(\si,\rd\si) \to (\mathcal{M},\rd\calM)$ be a free boundary minimal immersion of an oriented Riemannian manifold $\si$ with boundary $\rd\si$, and of dimension $n$. Then there is a constant $c = c(\mathcal{M}) > 0$, such that for any $\phi \in C^{1}(\si)$,
\begin{equation}\label{sobfree}
\left( \int_\Sigma |\phi|^{\frac{n}{n-1}}dA \right)^{\frac{n-1}{n}} \leq c\int_{\si}\bigl(|\nabla_{\si}\phi| + |\phi|\bigr)dA.
\end{equation}
\end{thm}

\begin{proof}
The condition $(\star)$ guarantees that we have an isometric immersion $w: \si \to \mathcal{N}$ of bounded mean curvature, where $\mathcal{N}$ is a Cartan-Hadamard manifold. So, we can use the Sobolev inequality of Hoffman-Spruck \cite{H.S}, for any function $\phi \in C^{1}(\si)$. Then, the proof is exactly like the proof of Theorem $2.2$ in \cite{E}.
\end{proof}

\begin{defin}\label{FBSI}
Given a Riemannian manifold $\mathcal{M}$ with boundary, we say that the free-boundary Sobolev inequality holds in $\mathcal{M}$, if the inequality \eqref{sobfree} is valid for all free-boundary minimal immersion $u:\si \to \mathcal{M}$ and all $\phi \in C^{1}(\si)$.
\end{defin}

\noindent
{\bf Remark}: The Theorem \ref{edel} shows that the free-boundary Sobolev inequality holds in every Riemannian manifold which satisfies the property $(\star)$.

\section{An estimate for the energy index}\label{sec5}

A reference for the definition and the properties of the heat kernels used in this section is the paper \cite{Gr}.

Let $u: (\si^{n},\rd\si^{n}) \to (\mathcal{M}^{n+k},\rd\calM^{n+k})$ be a free-boundary minimal immersion of an oriented compact manifold $\si$ with boundary $\rd\si$, into a Riemannian manifold $\mathcal{M}$ with boundary $\rd \mathcal{M}$, where $n\geq 2,\,k \geq 1$. Let $N$ be the outward pointing unit vector field orthogonal to the boundary of $\mathcal{M}$. We define the shape operator of $\rd \mathcal{M}$ by
$$S_{\rd \mathcal{M}}(X) = \nabla_{X}^{\calM}N.$$

Consider the rough laplacian $\Delta$ on $E = u^{*}(T\mathcal{M})$ defined by
$$\Delta V = \sum_{l=1}^{n} \bigl(\nabla_{e_l}\nabla_{e_l}V - \nabla_{\nabla_{e_l}e_l}V\bigr),$$
where $\{e_1,\ldots,e_n\}$ is a (locally defined) orthonormal frame field on $\si$. Also, given $\mathcal{F}, \mathcal{B} \in \ed(E)$, consider a quadratic form,
$$\mathcal{Q}(V,V) = \int_{\si} \ip{-\Delta V - \mathcal{F}(V)}{V} \, dA + \int_{\rd \si} \ip{\nabla_{\eta}V - \mathcal{B}(V)}{V}\, dL.$$

We say that $\lambda$ is an eigenvalue of $\mathcal{Q}$, if there exists $V \in \Gamma(E)$ such that $\mathcal{Q}(V,W) = \lambda\langle V,W\rangle_{L^2}, \forall\, W \in \Gamma(E)$, or equivalently
$$\left\{
\begin{array}{rl}
\Delta V + \mathcal{F}(V) + \lambda V = 0,\\\\
\nabla_{\eta}V - \mathcal{B}(V) = 0.
\end{array} \right.$$

In the remaining of this section we will assume that $V \in \Gamma(E)$ is an admissible variation with respect to $\rd \calM$. Let $\rho = \sup_{\mathcal{M}} |\mathcal{R}|$. We have
\begin{eqnarray}
\nonumber && \int_{\si} \ip{-\Delta V - \mathcal{R}(V)}{V} \, dA + \int_{\rd \si} \ip{\nabla_{\eta}V - S_{\rd \mathcal{M}}(V)}{V}\, dL\\ 
&\geq & \int_{\si} \ip{-\Delta V - \rho V}{V} \, dA + \int_{\rd \si} \ip{\nabla_{\eta}V  - S_{\rd \mathcal{M}}(V)}{V}\, dL.
\end{eqnarray}

Analysing this last inequality, we can conclude that if $V \in \Gamma(E)$ is such that the energy index form applied in $V$ is non-positive, then 
$$\mathcal{I}(V,V) \leq \rho\,|V|_{L^2}^{2},$$ 
where $\mathcal{I}$ is defined by
$$\mathcal{I}(V,V) := -\int_{\si} \ip{\Delta V}{V} \, dA + \int_{\rd \si} \ip{\nabla_{\eta}V - S_{\rd \mathcal{M}}(V)}{V}\, dL.$$
Denote by $\beta(\si)$ the number of eigenvalues of $\mathcal{I}$ less than or equal to $\rho$. It follows that
\begin{equation}\label{indb}
\ind_{\e}(u) + \n_{\e}(u) \leq \beta(\si).
\end{equation} 

In the following the covariant derivatives we will always be with respect to the $x$ variable. Consider the heat Kernel $K_{E}: \si\times\si\times (0,\infty) \to \ed(E)$ defined by 
$$\left\{
\begin{array}{rl}
\biggl(\displaystyle\frac{\rd}{\rd t} - \Delta\biggr) K_{E}(x,y,t) &=\ 0,  \quad \textrm{on} \ \si\times\si\times(0,\infty)\\\\
\displaystyle\lim_{t \to 0^+} K_{E}(x,y,t) &=\ \delta_{E}(x - y),  \quad \textrm{on} \ \si\times\si\\\\
\nabla_{\eta} K_{E}(x,y,t) - S_{\rd \mathcal{M}}\big(K_{E}(x,y,t)\big) &=\ 0,  \quad \textrm{on} \  \rd\si\times\si\times(0,\infty).
\end{array} \right.$$

Also, consider the heat Kernel $K: \si\times\si\times (0,\infty) \to \re$ on $\si$ given by
$$\left\{
\begin{array}{rl}
\biggl(\displaystyle\frac{\rd}{\rd t} - \Delta_{\si}\biggr) K(x,y,t) &=\ 0,  \quad \textrm{on} \ \si\times\si\times(0,\infty)\\\\
\displaystyle\lim_{t \to 0^+} K(x,y,t) &=\ \delta(x - y),   \quad \textrm{on} \ \si\times\si\\\\
\displaystyle\frac{\rd K}{\rd \eta}(x,y,t)  - \alpha K(x,y,t) &=\ 0,  \quad \textrm{on} \   \rd\si\times\si\times(0,\infty)
\end{array} \right.$$

\noindent
where $\Delta_{\si}$ is the laplacian acting on functions,
$$\alpha = \min\left\{0,\displaystyle\inf_{\rd \mathcal{M}}\inf_{|V| = 1} \langle S_{\rd \mathcal{M}}(V),V\rangle\right\},$$ and we suppose that $\alpha > -\infty$.
\\

Let $$\overline\lambda_1 \leq \overline\lambda_2 \leq \cdots \leq \overline\lambda_j \leq \cdots$$ 
be the spectrum of the eigenvalue problem of the rough laplacian with boundary condition $\nabla_{\eta} V - S_{\rd \mathcal{M}}(V) = 0$ and
$$\lambda_1 \leq \lambda_2 \leq \cdots \leq \lambda_j \leq \cdots$$ 
be the spectrum of the eigenvalue problem of the laplacian $\Delta_{\si}$ with boundary condition $\displaystyle\frac{\rd \phi}{\rd \eta} - \alpha \phi = 0$.

The traces $k_{E}$ and $k$ of $K_{E}$ and $K$ respectively, are given by
\begin{equation}
k_{E}(t) = \sum_{l = 1}^{\infty} e^{-\overline\lambda_l t}, \quad k(t) = \sum_{l = 1}^{\infty} e^{-\lambda_l t},\ t > 0.
\end{equation}
Thus, if $\overline\lambda_l \leq \rho$ is an eigenvalue of $\mathcal{I}$, we have $e^{-\rho t} \leq e^{-\overline\lambda_l t}, \forall t > 0$, hence
\begin{equation}\label{eingb}
\beta(\si)e^{-\rho t} \leq k_{E}(t), \ \forall t > 0.
\end{equation} 
So, to bound $\ind_{E}(\si) + \n_{\e}(u)$ it is sufficient to bound $k_{E}(t)$, which will be our main goal now. First, we will need some preliminary results about the kernels $K$ and $K_{E}$.\\

\begin{prop} The inequalities\label{lemine}
\begin{equation*}
\left\{
\begin{array}{rl}
\biggl(\displaystyle\frac{\rd}{\rd t} - \Delta_{\si}\biggr)| K_{E}| &\leq\ 0, \quad \textrm{on}\  \si\times\si\times(0,\infty).\\\\
\displaystyle\frac{\rd | K_{E}|}{\rd \eta} - \alpha| K_{E}| &\geq\ 0, \quad \textrm{on}\ \rd\si\times\si\times(0,\infty).
\end{array} \right.
\end{equation*}
hold in the sense of distributions.
\end{prop}
\begin{proof}
The proof of the first inequality is exactly like the one of proposition $2.2$ in \cite{H.S.U}. 

For the second inequality, given $\epsilon > 0$ define $$\phi_{\epsilon}(x,t) = \bigl(| K_{E}(x,y,t)|^2 + \epsilon^2\bigr)^{1/2}.$$
Observe that
\begin{equation}\label{boundine}
\frac{\rd \phi_{\epsilon}}{\rd \eta} = \frac{1}{\phi_{\epsilon}}\ib{\frac{\rd K_{E}}{\rd \eta}}{K_{E}} = \frac{1}{\phi_{\epsilon}}\ip{S_{\rd \mathcal{M}}(K_{E})}{K_{E}} \geq -\frac{\alpha}{\phi_{\epsilon}}| K_{E}|^2.
\end{equation}

\noindent
Letting $\epsilon \to 0$ on equation \eqref{boundine}, it follows that
\begin{equation*}
\frac{\rd | K_{E}|}{\rd \eta} - \alpha| K_{E}| \geq 0.
\end{equation*}
in the sense of distributions.
\end{proof}

\begin{lem}\label{lemcomp}
$|K_{E}(x,y,t)| \leq K(x,y,t), \ \forall t \geq 0.$
\end{lem}

\begin{proof}
By definition
\begin{equation}\label{in.con}
\lim_{t \to 0^+}\int_{\si}K(x,y,t)\phi(y)dA(y) = \phi(x) \ \ \textrm{and} \ \ \lim_{t \to 0^+}\int_{\si}K_{E}(x,y,t)\circ S_ydA(y) = S_x,
\end{equation}
which implies $| K_{E}(x,y,0)| = K(x,y,0)$. Moreover, using the fundamental theorem of calculus together with equation \eqref{in.con} we have
\begin{align*}
&\quad |K_{E}(x,y,t)| - K(x,y,t) = \int_{0}^{t}\frac{\rd}{\rd s}\int_{\si}| K_{E}(x,z,s)| K(z,y,t-s)dA(z)ds\\
&=\int_{0}^{t}\int_{\si}\biggl[\biggl(\frac{\rd}{\rd s}| K_{E}(x,z,s)|\biggr)K(z,y,t-s) + | K_{E}(x,z,s)|\frac{\rd}{\rd s}\bigl(K(z,y,t-s)\bigr)\biggr]dA(z)ds\\
&=\int_{0}^{t}\int_{\si}\biggl[\biggl(\frac{\rd}{\rd s}| K_{E}(x,z,s)|\biggr)K(z,y,t-s) - | K_{E}(x,z,s)|(\Delta_{\si})_{z}K(z,y,t-s)\biggr]dA(z)ds\\
&=\int_{0}^{t}\int_{\si}\biggl(\frac{\rd}{\rd s} - (\Delta_{\si})_{z}\biggr)\bigl(| K_{E}(x,z,s)|\bigr)K(z,y,t-s)dA(z)ds\\
&+ \int_{0}^{t}\int_{\rd\si}\biggl[| K_{E}(x,z,s)|\frac{\rd K}{\rd\eta}(z,y,t-s) - \frac{\rd | K_{E}|}{\rd\eta}(x,z,t)K(z,y,t-s)\biggr]dL(z)ds\\
&\leq  \int_{0}^{t}\int_{\rd\si}\alpha\biggl(| K_{E}|(x,z,t)K(z,y,t-s) - | K_{E}|(x,z,t)K(z,y,t-s)\biggr)dL(z)ds\\ 
&= 0,
\end{align*}
where is the last inequality we used the proposition \eqref{lemine} and the positivity of $K$.
\end{proof}

Its is easy to see that $\displaystyle\int_{\si}K(x,y,t)\, dA(y)$ is non-increasing in $t$. By equation \eqref{in.con} it follows that
\begin{equation}\label{mean}
\int_{\si}K(x,y,t)\, dA(y) \leq 1,\ \forall (x,t) \in \si\times(0,\infty).
\end{equation}

We can finally prove the main results of this section.\\

\begin{dem2}
We will use the methods of \cite{C.T}. Applying the Sobolev inequality for free boundary minimal surfaces \eqref{sobfree} to $\phi^2$ we have
\begin{equation*}\label{sob2}
\biggl(\int_{\si}\phi^4\, dA\biggr)^{1/2} \leq c\int_{\si}|\nabla_{\si} \phi^2|\, dA \ + \ c\int_{\si}\phi^2\, dA.
\end{equation*}

Using interpolation on the left hand side and the Holder inequality and the Arithmetic-Geometric inequality on the right hand side we obtain
\begin{equation}\label{mainine}
\biggl(\int_{\si}\phi^2\, dA\biggr)^{3/2}\biggl/\int_{\si}|\phi|\, dA \leq c_1\int_{\si}|\nabla_{\si} \phi|^2\, dA \ + \ c_2\int_{\si}\phi^2\, dA.
\end{equation}

Define $\phi(y) = K(x,y,t/2)$. We have 
\begin{equation*}
K(x,x,t) = \int_{\si}K^2(x,y,t/2)\, dA(y) = \int_{\si}\phi^2\, dA(y).
\end{equation*}
\noindent
So,
\begin{align}
\nonumber \frac{\rd K}{\rd t}(x,x,t) &= \int_{\si}K(x,y,t/2)\frac{\rd K}{\rd t}(x,y,t/2)\, dA(y)\\
\nonumber &= \int_{\si}K(x,y,t/2)\bigl((\Delta_{\si})_y K(x,y,t/2)\bigr)\, dA(y)\\
\nonumber &= - \int_{\si}|(\nabla_{\si})_y K(x,y,t/2)|^2\, dA(y)\\ \nonumber & \quad + \int_{\rd\si}K(x,y,t/2)\frac{\rd K}{\rd \eta}(x,y,t/2)\, dL(y)\\
\nonumber &= - \int_{\si}|(\nabla_{\si})_y K(x,y,t/2)|^2\, dA(y) - \alpha\int_{\rd\si}K^2(x,y,t/2)\, dL(y)\\
&\leq  - \int_{\si}|(\nabla_{\si})_y K(x,y,t/2)|^2\, dA(y).\label{der}
\end{align}

Substituting this in inequality \eqref{mainine}, and using the inequality \eqref{mean} it follows that
\begin{equation}
\bigl(K(x,x,t)\bigr)^{3/2} \leq - c_1\frac{\rd K}{\rd t}(x,x,t) + c_2 K(x,x,t).
\end{equation}

Let $\psi(t) = K^{-\frac{1}{2}}(x,x,t)$. Observe that $\psi(0) = 0$ $\bigl($because of equation \eqref{in.con}$\bigr)$. Multiplying both sides of the last inequality by $1/\bigl(2c_1\bigl(K(x,x,t)\bigr)^{3/2}\bigr)$ we obtain
 
\begin{equation*}
\frac{1}{2c_1}\leq \psi'(t) + \frac{c_2}{2c_1}\psi(t),
\end{equation*}
which implies,
\begin{equation*}
\psi(t) \geq \frac{1}{c_2}\bigl(1 - e^{-(c_{2}t/2c_{1})}\bigr).
\end{equation*}
Thus,
\begin{equation*}
K(x,x,t) \leq \frac{c_2^2}{\bigl(1 - e^{-(c_{2}t/2c_{1})}\bigr)^2}.
\end{equation*}

It follows that,
\begin{equation*}
k_{E}(t) \leq (n - 2)k(t) = (n - 2)\int_{\si}K(x,x,t)\,dA \leq \frac{(n- 2)c_2^2}{\bigl(1 - e^{-(c_{2}t/2c_{1})}\bigr)^2}|u(\si)|.
\end{equation*}
Therefore, using the inequalities \eqref{indb} and \eqref{eingb} we obtain
\begin{equation*}
\ind_{\e}(u) + \n_{\e}(u) \leq \min_{t>0}\Biggl(\frac{(n- 2)c_2^2e^{\rho t}}{\bigl(1 - e^{-(c_{2}t/2c_{1})}\bigr)^2}\Biggr)|u(\si)|.
\end{equation*}
\end{dem2}

Now, we will handle the case of higher dimensions. Consider an compact oriented manifold $\si^{n}$ with boundary $\rd\si$, and a Riemannian manifold $\mathcal{M}^{n+k}$ with boundary $\rd \mathcal{M}$, where $n \geq 3$ and $k\geq 1$. Let $u: (\si,\rd\si) \to (\mathcal{M},\rd\calM)$ be a free-boundary minimal immersion. Let $\Pi_1:E \to \si$ and $\Pi_2:F \to \rd\calM$ be Riemannian vector bundles such that $F|_{\rd\si} = E|_{\rd\si}$. Denote by $\Delta_{E}$ the rough laplacian of $E$, and let $\mathcal{S} \in \ed(E)$, $\mathcal{B} \in \ed(F)$. On this setting we have the following result.

\begin{thm}\label{taux}
Let $\si$, $\calM$ and $u$ as in the previous paragraph. Suppose that the free-boundary Sobolev inequality \eqref{sobfree} holds in $\mathcal{M}$ and that $|S_{\rd\calM}|$ is bounded. Then, there is a constant $c = c(n,\mathcal{M}) > 0$ such that the number of non-positive eigenvalues of $\Delta_E + \mathcal{S}$  with boundary condition $\nabla_{\eta}V - \mathcal{B}(V) = 0$ (denoted by $\beta_{E}$) satisfies
\begin{equation}
\beta_{E} \leq c\int_{\si}\Bigl(\max\bigl\{1,|\mathcal{S}|\bigr\}\Bigr)^{\frac{n}{2}}\,dA.
\end{equation}
\end{thm}

\begin{proof}
Reasoning as we did to prove the inequality \eqref{indb}, we conclude that it suffices to estimate the number of non-positive eigenvalues of the operator $\Delta_{\si} + q$ with boundary condition $\displaystyle\frac{\rd \phi}{\rd \eta} - \alpha \phi = 0$, where $q = |\mathcal{S}|$ and $\alpha = \min\left\{0,\displaystyle\inf_{\rd\calM}\inf_{|V| = 1} \langle\mathcal{B}(V),V\rangle\right\}$. The calculations are analogous to that of theorem $1$ in \cite{C.T} and are of the same spirit of the proof of the previous theorem, so for the sake of brevity we will only sketch the main steps.

Let $f = \max\{1,q\}$. Define $K_{E}$ as the kernel of $\displaystyle\frac{1}{p}\Delta_{E} - \frac{\rd}{\rd t}$ with boundary condition $\nabla_{\eta}K_{E} = \mathcal{B}(K_{E})$, and $K$ as the kernel of $\displaystyle\frac{1}{p}\Delta_{\si} - \frac{\rd}{\rd t}$ with boundary condition $\displaystyle\frac{\rd K}{\rd \eta} - \alpha K = 0$. On this setting, lemmas \ref{lemine} and \ref{lemcomp} are still valid. Let $\{\lambda_l\}_{l=0}^{\infty}$ be the eigenvalues of $\displaystyle\frac{1}{p}\Delta_{\si}$ with boundary condition $\displaystyle\frac{\rd \phi}{\rd \eta} - \alpha \phi = 0$. 

Define 
\begin{equation}\label{deftrace}
k(t) = \displaystyle\sum_{l=1}^{\infty}e^{-2\lambda_{l}t} = \int_{\si}\int_{\si} K^{2}(x,y,t)f(x)f(y)\ dA(y)dA(x).
\end{equation}
Arguing as we did to prove inequality \eqref{der} we obtain
\begin{equation}\label{der2}
\frac{dk}{dt} \leq -2\int_{\si}f(x)\biggl(\int_{\si}\bigl|(\nabla_{\si})_y K(x,y,t)\bigr|^2\ dA(y)\biggr)dA(x),
\end{equation}
and repeated applications of the H\"older inequality give us
\begin{equation}\label{bigine}
\begin{split}
k(t) \leq \ &\Biggl[\int_{\si}f(x)\biggl(\int_{\si}K^{\frac{2n}{n-2}}(x,y,t)\ dA(y)\biggr)^{\frac{n-2}{n}}dA(x)\Biggr]^{\frac{n}{n+2}}\\
\cdot &\Biggl[\int_{\si}f(x)\biggl(\int_{\si}K(x,y,t)f^{\frac{n+2}{4}}(y)\ dA(y)\biggr)^{2}dA(x)\Biggr]^{\frac{2}{n+2}}.
\end{split}
\end{equation}

Defining, $P(x,t) = \displaystyle\int_{\si}K(x,y,t)f^{\frac{n+2}{4}}(y)\ dA(y)$, we have
\begin{equation*}
\left\{
\begin{array}{rl}
\biggl(\displaystyle\frac{1}{p}(\Delta_{\si})_x - \displaystyle\frac{\rd}{\rd t}\biggr)P(x,t) &\leq\ 0, \quad \textrm{on}\ \si\times(0,\infty)\\\\
\displaystyle\frac{\rd P}{\rd \eta} - \alpha P &=\ 0, \quad \textrm{on}\ \rd\si\times(0,\infty)\\\\
P(x,0) &=\ f^{\frac{n-2}{4}}(x), \quad \textrm{on}\ \si.
\end{array} \right.
\end{equation*}

Again, as in the inequality \eqref{der} we obtain $\displaystyle\frac{d}{dt}\int_{\si}P^{2}(x,t)f(x)\ dA(x) \leq 0.$ Hence 
\begin{equation*}
\int_{\si}P^{2}(x,t)f(x)\ dA(x) \leq \int_{\si}P^{2}(0,t)f(x)\ dA(x) = \int_{\si}f^{\frac{n}{2}}(x)\, dA(x).
\end{equation*}
\noindent
Thus, the inequality \eqref{bigine} can be written as
\begin{equation*}\label{bigine2}
k^{\frac{n+2}{n}}(t)\biggl(\int_{\si}f^{\frac{n}{2}}(x)\ dA(x)\biggr)^{-2} \leq \int_{\si}f(x)\biggl(\int_{\si}K^{\frac{2n}{n-2}}(x,y,t)\ dA(y)\biggr)^{\frac{n-2}{n}}dA(x).
\end{equation*}

Choose $\phi(y) = K^{\frac{2(n-1)}{(n-2)}}(x,y,t)$ in the {\it free boundary Sobolev inequality} \eqref{sobfree}, square the inequality obtained and apply to inequality \eqref{bigine} to obtain
\begin{equation*}\label{bigine3}
\begin{split}
&\quad\ k^{\frac{n+2}{n}}(t)\biggl(\int_{\si}f^{\frac{n}{2}}(x)\, dA(x)\biggr)^{-2}\\
&\leq \int_{\si}f(x)\biggl(c_{1}\int_{\si}\bigl|(\nabla_{\si})_y K(x,y,t)\bigr|^2\, dA(y) + c_{2}\int_{\si} K^{2}(x,y,t)\, dA(y)\biggr)\\
&\leq -\frac{1}{2}c_{1}\frac{dk}{dt} + c_{2}k(t),
\end{split}
\end{equation*}
where on the last inequality we used equations \eqref{deftrace} and \eqref{der2}, and the fact that $f(y) \geq 1, \ \forall y \in \si$.

Following the method of the proof of the previous theorem to solve this differential inequality and using the inequalities \eqref{indb} and \eqref{eingb} we obtain
\begin{equation*}
\beta_{E} \leq \min_{t>0}\Biggl(\frac{c_{2}^{n/2}e^{2t}}{(1 - e^{-4c_{2}t/(nc_1)})^{n/2}}\Biggr)\int_{\si}f^{\frac{n}{2}}(x)\, dA.
\end{equation*}
\end{proof}

In the following, if $\phi$ is a function defined on $\mathcal{M}$ then we denote $\int_{\si}\phi\,dA := \int_{\si}\phi\circ u\, dA$. Applying this last result to the Jacobi operators of the enrgy functional and of the area functional we obtain the Theorem \ref{t3}.

In the following we continue in the setting of the previous theorem. We denote by $\beta_{p}(\si,\rd\si)$ the $p$-th relative Betti number of $\si$, by $\mathrm{II}_{\si}$ the second fundamental form of $\si$, and by $R^{\si}$ the curvature tensor of $\si$. Given $\xi \in N\si$, the shape operator $S_{\xi}:\Gamma(T\si) \to \Gamma(T\si)$ associated to $\xi$ is defined by
$$\langle S_{\xi}X,Y\rangle = -\langle \mathrm{II}_{\si}(X,Y),\xi\rangle.$$

We define
\begin{align*}
\mathcal{R}_{p}^{\si}(\omega)(X_1,\ldots,X_m) &:= \sum_{j=1}^{n}\sum_{l=1}^{p}(-1)^{l}\bigl(R^{\si}(e_j,X_l)\omega\bigr)(e_j,X_1,\ldots,\hat{X_l},\ldots,X_m),\\
S_{\xi}^{p}\omega\big(X_1,\ldots,X_m) &:= \sum_{j=1}^{p}\omega\bigl(X_1,\ldots,S_{\xi}X_m,\ldots,X_m\bigr),
\end{align*}
where $\omega$ is a $m$-form, $\{e_1,\ldots,e_n\}$ is a (locally defined) orthonormal frame field on $\si$, and $X_1,\cdots,X_m$ are tangent vectors of $\si$. Associated to the curvature tensor $R^{\calM}$ of $\calM$, we define $\mathcal{R}_{p}^{\calM}$ analogously to the case of $\si$.

Let $\xi_1\ldots,\xi_k$ be a (local) orthonormal frame for $N\si$. In \cite{Savo}, Savo obtained the following decomposition
\begin{align}\label{decomp}
\mathcal{R}_{p}^{\si} = \mathcal{R}_{p}^{\calM} + \sum_{j=1}^{k}\Big[S_{\xi_j}^{p}\circ S_{\xi_j}^{p}-(\mathrm{tr}S_{\xi_j}^{p})S_{\xi_j}^{p}\Big].
\end{align}

Since we assume the immersion $u$ is minimal, we have $\mathrm{tr}\,S_{\xi_j} = 0$, so we can write the second term on the right-hand side of the last equation as $$\mathcal{B}_{p} := \sum_{j=1}^{k}S_{\xi_j}^{p}\circ S_{\xi_j}^{p}.$$

\begin{rem}
On \cite{Savo}, the shape operator $S_{\xi}$ and the curvature tensor $R$ are both defined as the negative of ours, so we adapted the formulas accordingly with the changes of signs.
\end{rem}

\begin{cor}\label{Bochner}
Under the same assumptions of Theorem \eqref{taux}, there is a constant $c = c(n,\calM) > 0$ such that
\begin{equation*}
\beta_{p}(\si,\rd\si) \leq c\int_{\si}\Bigl(\max\bigl\{1,|\mathcal{R}_{p}^{\calM}| + |\mathcal{B}_{p}|\bigr\}\Bigr)^{\frac{n}{2}} dA.
\end{equation*}
\end{cor}

\begin{proof}
Consider the bundle of harmonic $m$-forms on $\si$ that are normal to the boundary, i.e,
$$\mathcal{H}_{T}^{p} = \{\omega \in \Omega^{p}(\si)/d\omega = 0, d^{*}\omega = 0 \ \textrm{and} \ \eta\wedge\omega = 0 \ \textrm{on} \ \rd\si\},$$
where $d$ denotes the exterior differential and $d^{*} = (-1)^{n(m+1)+1}*d\,*$. By the Bochner formula
\begin{equation}
0 = \bigl\langle(dd^{*} + d^{*}d)\omega,\omega\bigr\rangle = \bigl\langle -\Delta_{\si}\,\omega + \mathcal{R}_{p}^{\si}(\omega),\omega\bigr\rangle.
\end{equation}

Integrating by parts, it follows that
\begin{equation}
\int_{\si}\bigl(|\nabla_{\si}\,\omega|^2 + \langle\mathcal{R}_{p}^{\si}(\omega),\omega\rangle\bigr)dA - \int_{\rd\si}\langle\nabla_{\eta}\,\omega,\omega\rangle\, dL = 0.
\end{equation}

Moreover, $\langle\nabla_{\eta}\,\omega,\omega\rangle = -H_{\rd\calM}|\omega|^{2}$, see lemma $6$ in \cite{A.C.S3}. Therefore, an eigenvector associated to a zero eigenvalue of the operator $\Delta_{\si}\,\omega - \mathcal{R}_{p}(\omega)$ with boundary condition $\nabla_{\eta}\,\omega = H_{\rd\calM}\omega$, correspond to an harmonic $m$-form normal to the boundary. By the Hodge-de Rham theorem $\beta_{p}(\si,\rd\si) = \dim\mathcal{H}_{T}^{p}$. So applying the last theorem and using equation \eqref{decomp} we obtain the desired bound.
\end{proof}

\begin{rem}
Examining the proofs of the Sobolev inequality on \cite{E,H.S}, we see that the constants appearing on Theorems \ref{t2}, \ref{t3} and \ref{taux}, and on Corollary \ref{Bochner}, depend only on the dimension $n$ of $\si$ and the extrinsic curvature of $\partial\mathcal{M}$.
\end{rem}

\section{Applications}\label{sec6}

In this section $\mathcal{M}$ denotes a $3$-dimensional oriented compact Riemannian manifold with boundary $\rd \mathcal{M}$. Let $\si \subset \mathcal{M}$ be a minimal surface with boundary $\rd\si \subset \rd \mathcal{M}$, such that $\si$ is orthogonal to $\rd \mathcal{M}$ along $\rd\si$. We denote the area index and the area nullity of the inclusion map $\si \to \mathcal{M}$, by $\ind_{\ar}(\si)$ and $\n_{\ar}(\si)$ respectively. 

Following \cite{G.Z}, we say that $\si$ is {\it almost properly embedded}, if $\si$ is embedded and $\rd\si \subset \rd \mathcal{M}$. We denote by $\mathcal{S}(\Sigma)$ the touching set, $\mathrm{int}(\si)\cap \partial \mathcal{M}$, and define $\mathcal{R}(\Sigma) = \Sigma\setminus\mathcal{S}(\Sigma)$ as the {\em proper subset} of $\Sigma$. If $\mathcal{S}(\si) = \emptyset$, then we say that $\Sigma$ is \emph{properly embedded}; otherwise, we say that $\Sigma$ is \emph{improper}.

\subsection{Compactness}\label{sec6.1}

Define
\begin{equation*}
\begin{split}
&\mathfrak{M}_{\Lambda}  = \{\si \subset \mathcal{M} / \ \si \ \textrm{is a compact almost properly embedded}\\
&\textrm{free-boundary minimal surface, such that} \ |\si| + |\chi(\si)| \leq \Lambda\}.
\end{split}
\end{equation*}

We see from Theorem \ref{t2} that controlling the area and the topology of a free-boundary immersion we also control its index. So, combining \ref{t2} with the main results in \cite{A.C.S1,G.Z} we obtain:

\begin{thm}\label{compac}
Let $\mathcal{M}$ as above. For fixed $\Lambda > 0$, suppose that $\{\si_k\}$ is a sequence of orientable surfaces in $\mathfrak{M}_{\Lambda}$. Then there exists a finite set of points $\mathcal{Y} \subset \mathcal{M}$ and $\Sigma \subset \mathfrak{M}_{\Lambda}$ such that, up to a subsequence, $\Sigma_k$ converges smoothly and locally uniformly to $\Sigma$ on $\Sigma\setminus \mathcal{Y}$ with multiplicity $m \in \N$. Furthermore:
\begin{enumerate}
\item If $\si$ is orientable, then
\begin{enumerate}
\item $m=1$ if and only if $\mathcal{Y}=\emptyset$, and $\si_k\simeq \si$ eventually;
\item $m\geq 2$ if and only if $\mathcal{Y}\neq\emptyset$, and $\si$ is stable with $\n(\si) =1$. 
\end{enumerate}
\item If $\si$ is non-orientable and $\widetilde{\si}$ is its double cover, then $m\geq 2$ implies $\widetilde{\si}$ is stable, $\n(\widetilde{\si}) =1 $ and $\lambda_1(\si) >0$. In this case, $\mathcal{Y}=\emptyset$ implies $m=2$ and $\si_k\simeq \widetilde{\si}$ eventually. 
\end{enumerate}
\end{thm}

\begin{rem}
To prove the previous result in the properly embedded case we could alternatively use the main result of \cite{Wh}.
\end{rem}

\subsection{Index estimates}

As corollary of Theorem \ref{t1} we have the following index estimates.

\begin{prop}
Let $\mathcal{M}$ be a Riemannian manifold with $\partial\calM \neq \emptyset$. Suppose that $\mathcal{M}$ satisfies 
$$sec_\mathcal{M} \leq -\kappa \leq 0 \ \ \textrm{and} \ \ S_{\rd \mathcal{M}} \leq -\alpha\ip{\cdot}{\cdot} \leq 0.$$ Let $\si \subset \mathcal{M}$ be an immersed orientable compact free boundary minimal surface. Then
$$\ind_{\ar}(\si) + \n_{\ar}(\si) \leq \upsilon(g,m).$$
\end{prop}

\begin{proof}
Let $u:\si \to \mathcal{M}$ be the isometric immersion of $\si$ on $\mathcal{M}$. The conditions on $\mathcal{M}$ guarantee that $\ind_{\e}(u) = \n_{\e}(u) = 0$, so the result follows from Theorem \ref{nulcomp}.
\end{proof}

For interest results about free-boundary minimal surfaces in manifolds satisfying the properties of the previous theorem see \cite{E.R}. Also, is worth mentioning that by an easy adaptation of the proof of lemmas $4.1$ and $5.1$ in \cite{E.R}, we obtain the following area estimates, which combined with the previous result allow us to apply Theorem \ref{compac}.

\begin{prop}
Let $\mathcal{M}$ be a $3$-dimensional Riemannian manifold with $\partial\calM \neq \emptyset$. Suppose that $\mathcal{M}$ satisfies 
$$sec_\mathcal{M} \leq -\kappa < 0 \ \ \textrm{and} \ \ S_{\rd \mathcal{M}} \leq -\alpha\ip{\cdot}{\cdot} \leq 0.$$ Let $\si \subset \mathcal{M}$ be a orientable properly immersed compact free boundary minimal surface. Then
\begin{equation}
\kappa|\si| - \alpha|\rd\si| \leq -2\pi\chi(\si).
\end{equation}
Moreover, equality holds if, and only if, $\Sigma$ is totally geodesic, $K_{sect} \equiv -\kappa$ along $\Sigma$, $K_\Sigma = -\kappa$, and $k_{g} = -\alpha$, where $k_{g}$ is the geodesic curvature of $\rd\si$ as seen as a curve inside $\si$.
\end{prop}

Finally, for the case of Riemannian manifolds with nonnegative Ricci curvature and convex boundary we have Theorem \ref{t4}, whose proof we write below.

\begin{dem4}
It follows from Lemma $2.2$ and Theorem 3.1 of \cite{F.L} together with Theorem 2.3 of \cite{F.S1}, that there is a constant $c_1 = c_1(\mathcal{M}) > 0$ such that
\begin{align*}
|\si| &\leq c_1\frac{4\pi}{\alpha}(g + m).
\end{align*}

\noindent
Combining this with Theorems \ref{t1} and \ref{t2}, we obtain
\begin{align*}
\ind_{\ar}(\si) &\leq c_2(g + m) + \upsilon(g,m).
\end{align*}
Observe that we can bound the function $\upsilon(g,m)$ from above by a function which depends linearly of $g$ and $m$. In fact, if $\si$ is a disk then $\upsilon(g,m) = 0 < 1 = g+m$. If $\si$ is annulus, then $\upsilon(g,m) = 1 < 2 = g+m$. In all other cases we have
$$\upsilon(g,m) = 6g - 6 + 3m \leq 6(g + m).$$
Therefore the result follows.
\end{dem4}

\appendix

\section{The Riemann-Roch theorem for Riemann surfaces with boundary}\label{app}

In this section we state the version of the Riemann-Roch used in the text. We follow the appendix C of \cite{Mc.Sa}.

Let $\Pi: E \to \si$ be a complex vector bundle over a compact Riemann surface with boundary. Suppose $E$ is endowed with a hermitian structure $\ip{\cdot}{\cdot}$ and a complex structure $J : E \to E$. Given a subbundle $F \subset E|_{\rd\si}$, we say $F$ is \emph{totally real} if $F_p\perp JF_p,\, \forall\, p \in \rd\si$, and $F$ is of maximal (real) rank.

Denote by $W^{m,q}(E)$ the space of sections of $E$ of sobolev class $W^{m,q}$. We define
\begin{align*}
W_{F}^{m,q}(E) &= \{X \in W_{F}^{m,q}(E)/ X(\rd \si) \subset F\},\\
W_{F}^{m,q}(\Lambda^{0,1}\si\otimes E) &= \{\omega \in W_{F}^{m,q}(E)/ \omega(T\rd \si) \subset F\}.
\end{align*}
$$$$

Denote by $\bar{\rd} : C^{\infty}(\si,\C) \to \Lambda^{0,1}\si$ the operator obtained by composition of the exterior derivative $d : C^{\infty}(\si,\C) \to T^{*}\si\otimes\mathbb{C}$ with the projection $\Pi_{2}: \Lambda^{1,0}\si\oplus\Lambda^{0,1}\si \to \Lambda^{0,1}\si$.

\begin{defin}
A (complex linear, smooth) Cauchy-Riemann operator on the bundle $\Pi: E \to \si$ is a $\C$-linear operator
$$D: \Gamma(E) \to \Gamma(\Lambda^{0,1}\si\otimes_{\C}E)$$
which satisfies the Leibnitz rule
$$D(\phi X) = \phi(DX) + (\bar{\rd}\phi)X,$$
where $X \in \Gamma(E)$, $\phi \in C^{\infty}(\si,\C)$.
\end{defin}

\noindent
\begin{defin}
Let $l$ be a positive integer and $p > 1$ such that $lp > 2$. A real linear Cauchy-Riemann operator of class $W^{l-1,p}$ on $E$ is an operator of the form
$$D = D_0 + B,$$
where $B \in W^{l-1,p}\bigl(\Lambda^{0,1}\si\otimes\ed_{\re}(E)\bigr)$ and $D_0$ is a smooth complex linear Cauchy-Riemann operator on $E$.

\end{defin}

\begin{rem} 
Real linear Cauchy-Riemann operators satisfy the equation
$$D(\phi X) = \phi(DX) + (\bar{\rd}\phi)X,$$
only for real valued functions $\phi$.
\end{rem}

In this context we have the following theorem.

\begin{thm}[Riemann-Roch]\label{Riemroch}
Let $E \to \si$ be a complex holomorphic vector bundle over a compact Riemann surface with boundary, and $F \subset E|_{\rd\si}$ be a totally real subbundle, with $\dim_{\C}E = n$. Let $D$ be a real linear Cauchy-Riemann operator on $E$ of class $W_{F}^{l-1,p}$, where $l$ is a positive integer and $p > 1$ such that $lp > 2$. Then the following holds for every integer $k \in \{1,...,l\}$ and every real number $q > 1$ such that $k - 2/q \leq l - 2/p$:
\begin{enumerate}
\item The operators
\begin{eqnarray*}
D &:& W_{F}^{k,q}(E) \to W^{k-1,q}(\Lambda^{0,1}\si\otimes E),\\
D^{*}&:& W_{F}^{k,q}(\Lambda^{0,1}\si\otimes E) \to W^{k-1,q}(E),
\end{eqnarray*}
are Fredholm. Moreover, their kernels are independent of $k$ and $q$, and we have
\begin{equation*}
W \in \im(D) \Leftrightarrow \int_{\si}\ip{W}{W_0}\, dA, \ \forall \ W_0 \in \ker\bigl(D^{*}\bigr),
\end{equation*}
and
\begin{equation*}
V \in \im\bigl(D^{*}\bigr) \Leftrightarrow \int_{\si}\ip{V}{V_0}\, dA, \ \forall \ V_0 \in \ker(D).
\end{equation*}

\item The real Fredholm index of $\bar{\rd}$ is given by
\begin{equation}\label{fredind}
\ind{D} = n\chi(\si) + \mu(E,F),
\end{equation}

\item If $n = 1$ then
\begin{eqnarray}
\label{hyperb}\mu(E,F) < 0 \Rightarrow D \ \textrm{is injective},\\
\nonumber\\
\label{disk}\mu(E,F) + 2\chi(\si) > 0 \Rightarrow D \ \textrm{is surjective}.
\end{eqnarray}
\end{enumerate}
\end{thm}

\begin{rem} 
Here $\mu(E,F)$ denotes the boundary Maslov index of the pair $(E,F)$. We will not the define this invariant in all its generality (for this matter see section C.$3$ on the appendix C of \cite{Mc.Sa}), instead we will calculate it in the particular case needed here.
\end{rem}

\noindent
\label{ex}{\bf Example}: Let $\si$ be a Riemann surface with non-empty boundary and denote by $J$ its complex structure. Then its possible to endow the double $\widetilde\si$ of $\si$ with a complex structure $\widetilde{J}$ which is symmetric in the following sense: there exists an antiholomorphic diffeomorphism $S: \widetilde\si \to \widetilde\si$, such that $S^2 = Id$. If $(\widehat{\si},\widehat{J})$ is an exact duplicate of $(\si,J)$, then $S$ is defined by $S(x) = \widehat{x}, \ x \in \si$. For this construction see pages $264$ and $265$ of \cite{D.H.T}.

Consider the vector bundle $E = \Lambda^{0,1}\widetilde\si$. We have $E|_{\si} = \Lambda^{0,1}\si$, $E|_{\widehat{\si}} = \Lambda^{0,1}\widehat{\si}$. Denote $\gamma = \rd\si =\rd\widehat{\si}$, and consider the subbundle $F$ of $E|_{\gamma}$ whose sections in local isothermal coordinates are given by $f\partial_{\bar{z}}$, where $f$ is a real function. So, $F$ is a totally real subbundle.

By theorem $C.3.10$ in \cite{Mc.Sa} the first Chern number of $E$ satisfies
\begin{equation}
2\ip{c_{1}(E)}{[\widetilde\si]} = \mu(E|_{\si},F) + \mu(E|_{\widehat{\si}},F).
\end{equation}
On the other hand,
\begin{equation}
\ip{c_{1}(E)}{[\si]} = \chi(\widetilde\si) = 2\chi(\si).
\end{equation} 

\noindent
Moreover, since $E|_{\si}$ and $E|_{\widehat{\si}}$ are isomorphic, we have $\mu(E|_{\si},F) = \mu(E|_{\widehat{\si}},F)$. Therefore,
\begin{equation}
\mu(\Lambda^{0,1}\si,F) = 2\chi(\si).
\end{equation}

Now, suppose that $\si$ is endowed with a Hermitian metric, and extend it to all the bundles of tensors over $\si$. Consider the operator 
$$\bar{\rd}: W_{F}^{k,q}(\Lambda^{0,1}\si) \to W^{k-1,q}(\Lambda^{0,1}\si\otimes\Lambda^{0,1}\si).$$
Observe that $\Ker(\bar{\rd})$ coincides with the space of holomorphic sections of $T^{0,1}\si$ which are real on the boundary. We denote $\dim\Ker(\bar{\rd}) = 2h^{0}(\Lambda^{0,1}\si)$.

The metric gives rise to a trivialization of $\Lambda^{1,0}\si\otimes \Lambda^{0,1}\si$ and to a duality between $L^2(\Lambda^{0,1}\si\otimes\Lambda^{0,1}\si)$ and $L^2(T^{0,1}\si\otimes \Lambda^{1,0}\si)$. With respect to this duality, the adjoint of $\bar{\rd}$ is
$$-\bar{\rd}: W_{F}^{k,q}(T^{0,1}\si\otimes \Lambda^{1,0}\si) \to W^{k-1,q}(T^{0,1}\si\otimes \Lambda^{1,0}\si\otimes \Lambda^{0,1}\si),$$
so $\Ker(\bar{\rd})^*$ coincides with the space of holomorphic sections of $T^{0,1}\si\otimes \Lambda^{1,0}\si$ which are real on the boundary. Denote $\dim\Ker(\bar{\rd})^* = 2h^0(T^{0,1}\si\otimes \Lambda^{1,0}\si)$.

In view of all of this, the equation \eqref{fredind} on Theorem \ref{Riemroch} can be rewritten as
\begin{equation}
2h^{0}(\Lambda^{0,1}\si) - 2h^0(T^{0,1}\si\otimes \Lambda^{1,0}\si) = 3\chi(\si).
\end{equation}


\begin{thebibliography}{20}
\bibitem{A.C.S1} L. Ambrozio, A. Carlotto, B. Sharp. {\it Compactness of the space of minimal hypersurfaces with bounded volume and p-th Jacobi eigenvalue}. J. Geom. Anal. 26 (2016), no. 4, 2591-2601.

\bibitem{A.C.S3} L. Ambrozio, A. Carlotto, B. Sharp. {\it Index estimates for free boundary minimal hypersurfaces}. Math. Ann. 370 (2018), no. 3-4, 1063-1078.

\bibitem{A.C.S4} L. Ambrozio, A. Carlotto, B. Sharp. {\it Compactness analysis for free boundary minimal hypersurfaces}. Calc. Var. Partial Differential Equations 57 (2018), no. 1, Art. 22, 39 pp

\bibitem{B.S} R. Buzano, B. Sharp. {\it Qualitative and quantitative estimates for minimal hypersurfaces with bounded index and area}. Trans. Amer. Math. Soc. 370 (2018), 4373-4399

\bibitem{C} A. Carlotto. {\it Generic finiteness of minimal surfaces with bounded Morse index}. Ann. Sc. Norm. Super. Pisa Cl. Sci. (5) 17 (2017), no. 3, 1153-1171.

\bibitem{C.T}  S.Y. Cheng, J. Tysk. {\it Schr\"odinger operators and index bounds for minimal submanifolds}. Rocky Mountain J. Math. 24 (1994), no. 3, 977-996.

\bibitem{C.F} J. Chen, A. Fraser. {\it Holomorphic variations of minimal disks with boundary on a Lagrangian surface}. Canad. J. Math. 62 (2010), no. 6, 1264-1275.

\bibitem{C.F.P} J. Chen, A. Fraser, C. Pang. {\it Minimal immersions of compact bordered Riemann surfaces with free boundary}. Trans. Amer. Math. Soc. 367 (2015), no. 4, 2487-2507.

\bibitem{C.K.M} O. Chodosh, D. Ketover, D. Maximo. {\it Minimal hypersurfaces with bounded index}. Invent. Math. 209 (2017), no. 3, 617-664.

\bibitem{C} R. Courant. {\it The existence of minimal surfaces of given topological structure under prescribed boundary conditions}. Acta Math. 72 (1940), 51-98.

\bibitem{DL.R} C. De Lellis, J. Ramic. {\it Min-max theory for minimal hypersurfaces with boundary}.  Ann. Inst. Fourier (Grenoble) 68 (2018), no. 5, 1909-1986.

\bibitem{D.H.T} U. Dierkes, S. Hildebrandt. A. J. Tromba. {\it Global analysis of minimal surfaces}. Revised and enlarged second edition. Grundlehren der Mathematischen Wissenschaften [Fundamental Principles of Mathematical Sciences], 341. Springer, Heidelberg, 2010. xvi+537 pp. ISBN: 978-3-642-11705-3

\bibitem{D} B. Devyver. {\it Index of the critical catenoid}.  Geom. Dedicata 199 (2019), 355-371.

\bibitem{Do} J. Douglas. {\it Minimal surfaces of higher topological structure}. Ann. of Math. (2) 40 (1939) 205-298

\bibitem{E} N. Edelen. {\it Convexity estimates for mean curvature flow with free boundary.} Adv. Math. 294 (2016), 1-36.

\bibitem{E.M} N. Ejiri, M. Micallef. {\it Comparison between second variation of area and second variation of energy of a minimal surface}. Adv. Calc. Var. 1 (2008), no. 3, 223-239.

\bibitem{E.R} J. Espinar, H. Rosenberg. {\it Area estimates and rigidity of capillary $H-$surfaces in three-manifolds with boundary}.  Math. Z. 289 (2018), no. 3-4, 1261-1279. 

\bibitem{F.K} H. M. Farkas, I. Kra. {\it Riemann surfaces}. Second edition. Graduate Texts in Mathematics, 71. Springer-Verlag, New York, 1992. xvi+363 pp. ISBN: 0-387-97703-1

\bibitem{F1} A. Fraser. {\it On the free boundary variational problem for minimal disks}. Comm. Pure Appl. Math. 53 (2000), no. 8, 931-971.

\bibitem{F.L} A. Fraser, M. Li. {\it Compactness of the space of embedded minimal surfaces with free boundary in three-manifolds with nonnegative Ricci curvature and convex boundary}. J. Differential Geom. 96 (2014), no. 2, 183-200.

\bibitem{F.S1} A. Fraser, R. Schoen. {\it The first Steklov eigenvalue, conformal geometry, and minimal surfaces}. Adv. Math. 226 (2011), no. 5, 4011-4030.

\bibitem{F.S2} A. Fraser, R. Schoen. {\it Sharp eigenvalue bounds and minimal surfaces in the ball}. Invent. Math. 203 (2016), no. 3, 823-890.

\bibitem{Gr} P. Greiner. {\it An asymptotic expansion for the heat equation}. Arch. Rational Mech. Anal. 41 (1971), 163-218.

\bibitem{G.J} M. Gr\"uter and J. Jost. {\it On embedded minimal disks in convex bodies.
Ann. Inst. H. Poincar\'e}. Anal. Non Lin\'eaire 3 (1986), no. 5, 345-390.

\bibitem{G.Z} Q. Guang, X. Zhou. {\it Compactness and generic finiteness for free boundary minimal hypersurfaces}.  Pacific J. Math. 310 (2021), no. 1, 85-114. 

\bibitem{H.Z} H. Li, X. Zhou. {\it Existence of minimal surfaces of arbitrarily large Morse index}. Calc. Var. Partial Differential Equations 55 (2016), no. 3, Art. 64, 12 pp. 

\bibitem{H.S.U} H. Hess, R. Schrader, D. A. Uhlenbrock. {\it Kato's inequality and the spectral distribution of Laplacians on compact Riemannian manifolds.}  J. Differential Geom. 15 (1980), no. 1, 27–37 (1981).

\bibitem{H.S} D. Hoffman, J. Spruck. {\it Sobolev and isoperimetric inequalities for Riemannian submanifolds}. Comm. Pure Appl. Math. 27 (1974), 715-727.

\bibitem{J2} J. Jost. {\it Two-dimensional geometric variational problems}. J. Wiley and Sons, Chichester, N.Y. (1991). MR1100926 (92h:58045)

\bibitem{K.L} N. Kapouleas, M. Li. {\it Free boundary minimal surfaces in the unit three-ball via desingularization of the critical catenoid and the equatorial disc}. J. Reine Angew. Math. 776 (2021), 201-254. 

\bibitem{K.W} N. Kapouleas, D. Wiygul. {\it Free-boundary minimal surfaces with connected boundary in the 3-ball by tripling the equatorial disc}. arXiv:1711.00818, to appear in Journal of Differential Geometry.

\bibitem{K} D. Ketover. {\it Free boundary minimal surfaces of unbounded genus}. arXiv:1612.08692 [math.DG]

\bibitem{L.Y} Peter Li and Shing-Tung Yau. {\it On the Schr\"odinger equation and the eigenvalue problem}. Comm. Math. Phys., 88, no.3 (1983), 309-318.


\bibitem{L.Z} M. Li, X. Zhou. {\it Min-max theory for free boundary minimal hypersurfaces I - regularity theory}.  J. Differential Geom. 118 (2021), no. 3, 487-553.

\bibitem{M.N} F.C. Marques, A. Neves. {\it Existence of infinitely many minimal hypersurfaces in positive Ricci curvature}. Invent. Math. 209 (2017), no. 2, 577-616.

\bibitem{M.N.S} D. M\'aximo, I. Nunes, G. Smith. {\it Free boundary minimal annuli in convex three-manifolds}. J. Differential Geom. 106 (2017), no. 1, 139-186.

\bibitem{Mc.Sa} D. McDuff, D. Salamon. {\it J-holomorphic curves and symplectic topology}. Second edition. American Mathematical Society Colloquium Publications, 52. American Mathematical Society, Providence, RI, 2012. xiv+726 pp. ISBN: 978-0-8218-8746-2

\bibitem{M.Y} W. Meeks and S.T. Yau. {\it Topology of three-dimensional manifolds and the embedding problems in minimal surface theory}. Ann. of Math. (2) 112 (1980), no. 3, 4414-84.


\bibitem{F.P.Z} A. Folha, F. Pacard, T. Zolotareva. {\it Free boundary minimal surfaces in the unit 3-ball}.  Manuscripta Math. 154 (2017), no. 3-4, 359-409.

\bibitem{S.U1} J. Sacks, K. Uhlenbeck. {\it The existence of minimal immersions of
2-spheres}. Ann. of Math. (2) 113 (1981) 1-24.

\bibitem{S.U2} J. Sacks, K. Uhlenbeck. {\it Minimal immersions of closed Riemann surfaces}. Trans. Amer. Math. Soc. 271 (1982), no. 2, 639-652. 

\bibitem{S} P. Sargent. {Index bounds for free boundary minimal surfaces of convex bodies}. Proc. Amer. Math. Soc. 145 (2017), no. 6, 2467-2480.

\bibitem{Savo} A. Savo, {\it The Bochner formula for isometric immersions}. Pacific J. Math. 272 (2014), no. 2, 395-422.

\bibitem{Sh} B. Sharp. {\it Compactness of minimal hypersurfaces with bounded index}. J. Differential Geom. 106 (2017), no. 2, 317-339.

\bibitem{S.Y} R. Schoen, S.T. Yau. {\it Existence of incompressible minimal surfaces and the topology of three-dimensional manifolds with nonnegative scalar curvature}. Ann. of Math. (2) 110 (1979), no. 1, 127-142. 

\bibitem{S.Z} G. Smith, D. Zhou. {\it The Morse index of the critical catenoid}.  Geom. Dedicata 201 (2019), 13-19.

\bibitem{S.S.T.Z} G. Smith, A. Stern, H. Tran, D. Zhou. {\it On the Morse index of higher-dimensional free boundary minimal catenoids}. arXiv:1709.00977 [math.DG]

\bibitem{St} M. Struwe. {\it On a free boundary problem for minimal surfaces}. Invent. Math. 75 (1984), no. 3, 547-560.

\bibitem{T} H. Tran. {\it Index Characterization for Free Boundary Minimal Surfaces}.  Comm. Anal. Geom. 28 (2020), no. 1, 189-222.

\bibitem{T.T} F. Tomi, A.J. Tromba. {\it Existence theorems for minimal surfaces of non-zero genus spanning a contour}. Mem. Amer. Math. Soc. 71 (1988) Number 382

\bibitem{W} R. O. Wells, Jr. {\it Differential Analysis on Complex Manifolds}. Third edition. With a new appendix by Oscar Garcia-Prada. Graduate Texts in Mathematics, 65. Springer, New York, 2008. xiv+299 pp. ISBN: 978-0-387-73891-8

\bibitem{Wh} B. White, {\it On the compactness theorem for embedded minimal surfaces in 3-manifolds with locally bounded area and genus}. Comm. Anal. Geom. 26 (2018), no. 3, 659-678. 

\bibitem{Z2} X. Zhou. {\it On the existence of min-max minimal surface of genus $g \geq 2$}. Commun. Contemp. Math. 19 (2017), no. 4, 1750041, 36 pp.

\end{thebibliography}
\end{document}